\documentclass{article}
\usepackage[utf8]{inputenc}
\usepackage[T1]{fontenc}
\usepackage{lineno,hyperref}
\usepackage{mathtools}
\usepackage{verbatim}
\usepackage{graphicx}
\usepackage{rotating}

\usepackage{adjustbox}%for rescaling tikzcd diagrams

\usepackage{tikz}
\usepackage{tikz-cd}
\usetikzlibrary{backgrounds,circuits,circuits.ee.IEC,shapes,fit,matrix}

\tikzstyle{simple}=[-,line width=2.000]
\tikzstyle{arrow}=[-,postaction={decorate},decoration={markings,mark=at position .5 with {\arrow{>}}},line width=1.100]
\pgfdeclarelayer{edgelayer}
\pgfdeclarelayer{nodelayer}
\pgfsetlayers{edgelayer,nodelayer,main}

\tikzstyle{none}=[inner sep=0pt]

\tikzstyle{species}=[circle,fill=yellow,draw=black,scale=1.15]
\tikzstyle{transition}=[rectangle,fill=lblue,draw=black,scale=1.15]
\tikzstyle{inarrow}=[->, >=stealth, shorten >=.03cm,line width=1.5]
\tikzstyle{empty}=[circle,fill=none, draw=none]
\tikzstyle{inputdot}=[circle,fill=purple,draw=purple, scale=.25]
\tikzstyle{inputarrow}=[->,draw=purple, shorten >=.05cm]
\tikzstyle{simple}=[-,draw=purple,line width=1.000]

\usepackage{xspace}
\usepackage{pslatex}

\usepackage{datetime}
\newdateformat{versiondate}{%
\THEMONTH\THEDAY}

\usepackage{amsthm}
\usepackage{amsmath} % flush equations left
\usepackage{amsfonts}
\usepackage{mathrsfs} 
\usepackage{amssymb}

\usepackage{mathtools}

\usepackage[inline]{enumitem}
\usepackage{amsmath}
\usepackage{framed} %for framing of computational problems 

\usepackage{tabularx,booktabs}
\usepackage{xltabular}
\usepackage{subcaption}
\def\DEBUG{true} %comment out to get rid of colors
\ifdefined\DEBUG{}

\fi

\usepackage{pgf,tikz,environ}
\usepackage{quiver}
\usetikzlibrary{arrows}
\usetikzlibrary{cd}
\usetikzlibrary{positioning}
\usetikzlibrary{fadings}
\usetikzlibrary{decorations.pathreplacing}
\usetikzlibrary{decorations.pathmorphing}
\tikzset{snake it/.style={decorate, decoration=snake}}

\definecolor{parchment}{RGB}{214, 204, 169}

\theoremstyle{plain}%-----------------------

\newtheorem{theorem}{Theorem}[section]

\newtheorem{lemma}[theorem]{Lemma}

\newtheorem{proposition}[theorem]{Proposition}

\theoremstyle{definition}
\newtheorem{definition}[theorem]{Definition}

\newtheorem{notation}[theorem]{Notation}

\newtheorem{note}[theorem]{Note}
\newtheorem{aside}[theorem]{Aside}

\usepackage{mdframed} %for shaded theorems and definitions
\surroundwithmdframed[hidealllines=true,backgroundcolor=parchment!40]{theorem}
\surroundwithmdframed[hidealllines=true,backgroundcolor=parchment!40]{corollary}
\surroundwithmdframed[hidealllines=true,backgroundcolor=parchment!40]{lemma}
\surroundwithmdframed[hidealllines=true,backgroundcolor=parchment!40]{problem}
\surroundwithmdframed[hidealllines=true,backgroundcolor=parchment!40]{example}
\surroundwithmdframed[hidealllines=true,backgroundcolor=parchment!40]{aside}
\surroundwithmdframed[hidealllines=true,backgroundcolor=parchment!40]{proposition}
\surroundwithmdframed[hidealllines=true,backgroundcolor=parchment!40]{conjecture}
\surroundwithmdframed[hidealllines=true,backgroundcolor=parchment!40]{definition}
\surroundwithmdframed[hidealllines=true,backgroundcolor=parchment!40]{notation}
\surroundwithmdframed[hidealllines=true,backgroundcolor=parchment!40]{observation}
\surroundwithmdframed[hidealllines=true,backgroundcolor=parchment!40]{claim}

\newcommand{\define}[1]{\textbf{#1}}

\newcommand{\cat}[1]{\mathsf{#1}}
\newcommand{\typesetcategory}[1]{\cat{#1}}

\DeclareMathOperator{\Cu}{\typesetcategory{Cu(T,D)}}

\DeclareMathOperator*{\colim}{colim}

\DeclareMathOperator{\Fa}{\mathcal{F}}  
\DeclareMathOperator{\Ga}{\mathcal{G}}

\DeclareMathOperator{\Ka}{\mathcal{K}}

\DeclareMathOperator{\Oa}{\mathcal{O}}  
\DeclareMathOperator{\Pa}{\mathcal{P}}

\DeclareMathOperator{\Ta}{\mathcal{T}}

\DeclareMathOperator{\Xa}{\mathcal{X}}

\usepackage[margin = 1.2in]{geometry}
\makeatletter
\def\@fnsymbol#1{\ensuremath{\ifcase#1\or
    *\or
    \dagger\or
    \ddagger\or
    \mathsection\or
    {\vdash}\or%\mathparagraph\or
    \else\@ctrerr\fi}}
\makeatother

\title{\fontfamily{lmss}\selectfont Towards a Unified Theory of Time-Varying Data}
\author{
    Benjamin Merlin Bumpus~\thanks{(\textit{Corresponding authors}.) \newline \indent{\;\;} University of Florida, Computer \& Information Science \& Engineering, Florida, USA.} 
    \and
    James Fairbanks\footnotemark[1]%~\thanks{University of Florida, Computer \& Information Science \& Engineering, Florida, USA. }
    \and
    Martti Karvonen~\thanks{University of Ottawa, Department of Mathematics, Canada.}
    \and
    Wilmer Leal\footnotemark[1]%\thanks{University of Florida, Computer \& Information Science \& Engineering, Florida, USA. }
    \and
    Frédéric Simard\footnotemark[3]~\thanks{University of Ottawa, School of Electrical Engineering and Computer Science, Canada.}
}
\date{Last compilation: \today}

\begin{document}

\maketitle

\newcommand\blfootnote[1]{%
  \begingroup
  \renewcommand\thefootnote{}\footnote{#1}%
  \addtocounter{footnote}{-1}%
  \endgroup
}

\begin{abstract}
What is a time-varying graph, a time-varying topological space, or, more generally, a mathematical structure that evolves over time? In this work, we lay the foundations for a general theory of temporal data by introducing \textit{categories of narratives}. These are sheaves on posets of time intervals that encode snapshots of a temporal object along with the relationships between them. This theory satisfies five desiderata distilled from the burgeoning field of time-varying graphs:
\begin{enumerate*}[label=\textbf{(D\arabic*)}]
\item it defines both time-varying objects and their morphisms;
\item it distinguishes between cumulative and persistent interpretations and provides principled methods for transitioning between them;
\item it systematically lifts static notions to their temporal analogues;
\item it is object agnostic;
\item it integrates with theories of dynamical systems.
\end{enumerate*}
To achieve this, we build upon existing categorical and sheaf-theoretic approaches to temporal graph theory, generalizing them to any category with limits and colimits. We also formalize tacit intuitions that, while present, often remain implicit in temporal graph theory. Beyond synthesizing and reformulating existing ideas in categorical language, we introduce sheaf-theoretic constructions and prove results that, to our knowledge, have not appeared in the temporal data literature—such as the adjunction between persistent and cumulative narratives. More importantly, we integrate these existing and novel elements into a consistent and coherent framework, setting the stage for a unified theory of time-varying data.    
\end{abstract}

%%%%%%%%%%%%%%%%%%%%%%%%%%%

%%%%%%%%%%%%%%% End of first page %%%%%%%%%%%%%%%%%%%%%

\maketitle
\section{Introduction}\label{sec:intro}
We can never fully observe the underlying dynamics that govern nature. Instead we are left with two approaches; we call these: the `method of axioms' and `method of data'. The first focuses on establishing mechanisms (specified via, for example, differential equations or automata) that align with our experience of the hidden dynamics we seek to study. 
On the other hand, the `method of data' emphasizes empirical observations, discerning appropriate mathematical structures that underlie the observed time-varying data and extracting meaningful insights into the time-varying system. Both of these approaches are obviously interlinked, yet the lack of a formal, unifying framework for time-varying data structures prevents us from making this connection explicit. While significant progress has been made in fields such as graph theory and topological data analysis, a more encompassing theoretical framework is needed—one that unifies these perspectives and extends to other temporal data structures.

In studying the data we can collect over time, we limit ourselves to the `visible' aspects of the underlying dynamics. Thus, in much the same way as one can glean some (but perhaps not much) of the narrative of Romeo and Juliet by only reading a page of the whole, we view time-varying data as an observable narrative that tells a small portion of larger stories governed by more complex dynamics. This simple epistemological stance appears implicitly in many areas of mathematics concerned with temporal or time-varying data. For instance, consider the explosive birth of temporal graph theory. Here, one is interested in graphs whose vertices and edges may come and go over time. To motivate these models, one tacitly appeals to the connection between time-varying data and a hidden dynamical system that generates it. A common example in the field of temporal graphs is that of opportunistic mobility~\cite{FlocchiniQuattrocchiSantoroCasteigts}: physical objects in motion, such as buses, taxis, trains, or satellites, transmit information between each other at limited distances, and snapshots of the communication networks are recorded at various evenly-spaced instants in time. Further examples that assume the presence of underlying dynamics include human and animal proximity networks, human communication networks, collaboration networks, citation networks, economic networks, neuro-scientific networks, biological, chemical, ecological, and epidemiological networks~\cite{HARARY199779, OthonTempSurvey, HOLME201297, leal2022exploration, holme2015modern, FlocchiniQuattrocchiSantoroCasteigts}.

Although it is clear that what makes data \textit{temporal} is its link to an underlying dynamical system, this connection is in no way mathematically explicit and concrete. Indeed one would expect there to be further mathematical properties of temporal data which allow us to distinguish a mere \(\mathbb{N}\)-indexed sequence of sets or graphs or groups, say, from their temporal analogues. As of yet, though, this distinction has rarely been formally and systematically investigated. For example think of temporal graphs once again. Modulo embellishing attributes such as latencies or wait times, \textit{typical} definitions 
\footnote{As an example of a notable exception, see Kim and Mémoli~\cite{Kim2023}. For a more detailed discussion on how morphisms are considered in \cite{Kim2023}, see points one and two of Section~\ref{sec:desiderata}.} simply require temporal graphs to be sequences of graphs (e.g.~\cite{HARARY199779, OthonTempSurvey, HOLME201297, holme2015modern, FlocchiniQuattrocchiSantoroCasteigts, KEMPE2002820, deleting_edge_temporal_graphs_Kitty_and_Jess, assigning_times_restrict_disease_spread, jess_temp_gr_livestock_multispecies, jess_temporal_epidemics, Group_Formation_Backstrom, llanos19, Influence_social_networks_kempe, Kempe_distrib_systems}). There no further semantics on the relationships between time steps is imposed. And these definitions never explicitly state what kind of global information should be tracked by the temporal data: is it the total accumulation of data over time or is it the persistent structure that emerges in the data throughout the evolution of the underlying dynamical system? 

\paragraph{In this paper} we ask: \textit{how does one build a robust and general theory of temporal data?} To address this question, we first turn to the theory of time-varying graphs, a field that has received considerable attention in recent years~\cite{HARARY199779, OthonTempSurvey, HOLME201297, holme2015modern, FlocchiniQuattrocchiSantoroCasteigts, KEMPE2002820, deleting_edge_temporal_graphs_Kitty_and_Jess, assigning_times_restrict_disease_spread, jess_temp_gr_livestock_multispecies, jess_temporal_epidemics, Group_Formation_Backstrom, llanos19, Influence_social_networks_kempe, Kempe_distrib_systems}. This body of work has served both as a source of inspiration and as a foundation for our approach, offering concrete ideas, methodologies and perspectives that have shaped our understanding of time-varying structures. By analyzing the questions and techniques developed in this field, we seek to construct an abstract framework that captures the essential structures and hypotheses necessary for generalizing temporal data analysis across a broad spectrum of mathematical structures beyond graphs.

One key hypothesis that emerges from the study of time-varying graphs is that temporal data should not only record what occurred at various instants in time, but it should also keep track of the relationships between successive time-points.  In other words, what makes data \textit{temporal} is whether it is `in the memory'~\cite{sep-time-experience} in the sense of st Augustine's Confessions~\cite{augustine2016confessionsVol-1, augustine2016confessionsVol-2}. Hidden in this seemingly simple statement, is the structure of a \textit{sheaf}: a temporal set (or graph or group, etc.) should consist of an assignment of a data set at each time point together with \textit{consistent} assignments of sets over each \textit{interval of time} in such a way that the sets assigned on intervals are determined by the sets assigned on subintervals. The sheaf-theoretic perspective we adopt here builds upon Schultz, Spivak and Vasilakopoulou's~\cite{schultz2017temporal} notion of an interval sheaf and it allows for a very general definition of temporal objects. 

\paragraph{Our contribution} is twofold; first we distill the lessons learned from temporal graph theory into the following set of \textit{desiderata} for any mature theory of temporal data: 
\begin{enumerate}[label=\textbf{(D\arabic*)}]
    \item (\textbf{Categories of Temporal Data}) Any theory of temporal data should define not only time-varying data, but also appropriate morphisms thereof.\label{desideratum-1}
    \item (\textbf{Cumulative and Persistent Perspectives}) In contrast to being a mere sequence, temporal data should explicitly record whether it is to be viewed cumulatively or persistently. Furthermore there should be methods of conversion between these two viewpoints.\label{desideratum-2}  
    \item (\textbf{Systematic `Temporalization'}) Any theory of temporal data should come equipped with systematic ways of obtaining temporal analogues of notions relating to static data.\label{desideratum-3} 
    \item (\textbf{Object Agnosticism}) Theories of temporal data should be object agnostic and applicable to any kinds of data originating from given underlying dynamics.\label{desideratum-4} 
    \item (\textbf{Sampling}) Since temporal data naturally arises from some underlying dynamical system, any theory of temporal data should be seamlessly interoperable with theories of dynamical systems.\label{desideratum-5} 
\end{enumerate}
Our second main contribution is to introduce \textit{categories of narratives}, an object-agnostic theory of time-varying objects which satisfies the desiderata mentioned above. 
As a benchmark, we then observe how standard ideas of temporal graph theory crop up naturally when our general framework of temporal objects is instantiated on graphs. 

We choose to see this task of theory-building through a category theoretic lens for three reasons. First of all this approach directly addresses our first desideratum \ref{desideratum-1}, namely that of having an explicit definition of \textit{isomorphisms} (or more generally \textit{morphisms}) of temporal data. Second of all, we adopt a category-theoretic approach because its emphasis, being not on objects, but on the \textit{relationships} between them~\cite{riehl2017category, awodey2010category}, makes it particularly well-suited for general, object-agnostic definitions. Thirdly, \textit{sheaves}, which are our main technical tool in the definition of time-varying data, are most naturally studied in category theoretic terms~\cite{rosiak-book, maclane2012sheaves}. 

\subsection{Previous and Related Work: Accumulating Desiderata for a General Theory of Temporal Data} \label{sec:desiderata}

There are as many different definitions of temporal graphs as there are application domains from which the notion can arise. This has lead to a proliferation of many subtly different concepts such as: \emph{temporal graphs}, \emph{temporal networks}, \emph{dynamic graphs}, \emph{evolving graphs} and \emph{time-varying graphs}~\cite{HARARY199779, OthonTempSurvey, HOLME201297, holme2015modern, FlocchiniQuattrocchiSantoroCasteigts, KEMPE2002820}. Each model of temporal graphs makes different assumptions on what may vary over time. For example, are the vertices fixed, or may they change? Does it take time to cross an edge? And does this change as an edge appears and disappears? If an edge reappears after having vanished at some point in time, in what sense has it returned, is it the same edge?

The novelty of these fields and the many fascinating directions for further inquiry they harbor make the mathematical treatment of temporal data exciting. However, precisely because of the field's youth, we believe that it is crucial to pause and distill the lessons we have learned from temporal graphs into desiderata for the field of temporal data more broadly. In what follows we shall briefly contextualize each desideratum mentioned above in turn while also signposting  previous work and how our framework addresses each point. We begin with \ref{desideratum-1}.  

\begin{enumerate}

%Morphisms of temporal graphs are rarely taken as a central notion\footnote{In \cite{Kim2023}, for example, morphisms between temporal graphs are not directly mentioned, let alone used: the authors comment in [Table 3]\cite{Kim2023} that cosheaves (which come equipped with a notion of morphism) valued in the poset of subobjects of a fixed complete graph could be understood as temporal graphs.}, let alone defined formally in the first place. 

\item Morphisms of temporal graphs are rarely treated as a central concept, let alone formally defined and systematically exploited to gain insights or prove new results. In~\cite{Kim2023}, for example, the authors note in [Table 3]\cite{Kim2023} that cosheaves in the poset of subobjects of a fixed complete graph can be interpreted as cumulative temporal graphs. This is a beautiful idea that we will systematically investigate, extend, and connect with other aspects of temporal data (see Section \ref{sec:background}). However, morphisms are only implicitly recognized in~\cite{Kim2023} (since cosheaves inherently come with a notion of morphism), and their potential for yielding technical results and deeper insights into temporal structures remains largely unexplored in both this work and the broader literature. This is a serious impediment to the generalization of the ideas of temporal graphs to other time-varying structures since any such general theory should be invariant under isomorphisms. Thus we distill our first desideratum \ref{desideratum-1}: theories of temporal data should not only concern themselves with what time-varying data is, but also with what an appropriate notion of \textit{morphism} of temporal data should be. %Moreover, this notion should be central to such a theory.
These morphisms will enable analysis of temporal objects.

Narratives, our definition of time-varying data (Definition~\ref{def:narratives}), are stated in terms of certain kinds of sheaves. This immediately addresses desideratum~\ref{desideratum-1} since it automatically equips us with a suitable and well-studied~\cite{rosiak-book, maclane2012sheaves} notion of a morphism of temporal data, namely \textit{morphisms of sheaves}. Then, by instantiating narratives on graphs in Section~\ref{sec:examples-of-narratives}, we define \textit{categories} of temporal graphs as a special case of the broader theory.  This notion of morphisms is central to accomplishing \textit{each} of the desiderata.  This is especially true for Theorem \ref{thm:adjunction}, where we define an adjunction between the persistent and cumulative perspectives \ref{desideratum-2}; and also in Section \ref{sec:lifting-properties}, where classes of temporal objects are defined in terms of which morphisms of temporal objects they admit \ref{desideratum-3}.      

\item Our second desideratum is born from observing that most current definitions of temporal graphs are equivalent to mere sequences of graphs~\cite{FlocchiniQuattrocchiSantoroCasteigts, KEMPE2002820} (snapshots) without explicit mention of how each snapshot is related to the next (see below for a notable exception in TDA). To understand the importance of this observation, we must first note that in any theory of temporal graphs, one always finds great use in relating time-varying structure to its older and more thoroughly studied \textit{static} counterpart. For instance, any temporal graph is more or less explicitly assumed to come equipped with an \textit{underlying static graph}~\cite{FlocchiniQuattrocchiSantoroCasteigts, KEMPE2002820}. This is a graph consisting of all those vertices and edges that were ever seen to appear over the course of time and it should be thought of as the result of \textit{accumulating} data into a static representation. Rather than being presented as part and parcel of the temporal structure, the underlying static graphs are presented as the result of carrying out a computation—that of taking unions of snapshots—involving input temporal graphs. The implicitness of this representation has two drawbacks. The first is that it does not allow for vertices or edges to \textit{merge} or \textit{divide} over time; these are very natural operations that one should expect of time-varying graphs in the `wild' (think for example of cell division or acquisitions or merges of companies). The second drawback of the implicitness of the computation of the underlying static graph is that it conceals another very natural static structure that always accompanies any given temporal graph, we call it the \textit{persistence graph}. This is the static graph consisting of all those vertices and edges which persisted throughout the entire life-span of the temporal graph. We distill this general pattern into desideratum~\ref{desideratum-2}: temporal data should come explicitly equipped with either a \textit{cumulative} or a \textit{persistent} perspective which records which information we should be keeping track of over intervals of time.

Outside of the temporal graph community, in topological data analysis the authors of \cite{Kim2023} use cosheaves over intervals to define a time-varying graph whose snapshots are all subgraphs of a fixed, globally-defined complete graph. In our language, this is an instance of a cumulative narrative valued in the subobject poset of $K_n$ for some $n$. Kim and Memoli~\cite{Kim2023} are interested in applications of this cosheaf perspective for uses in TDA and do not consider persistent time-varying graphs (sheaves over intervals) and the relationship between these two kinds of temporal graphs. Our work differs from \cite{Kim2023} in three fundamental ways: (1) we define cumulative narratives not just for subobject posets of the complete graphs, but for any category with colimits\footnote{Notice that the approach in~\cite{Kim2023} does not allow merging of vertices or edges since the corestriction maps of such cosheaves must necessarily be injections. In contrast, our definition of cumulative time-varying graphs as cosheaves valued $\cat{Grph}$ (the category of graphs and their homomorphisms) allows for merging and splitting of vertices and edges over time.}; (2) we introduce persistent time-varying graphs and, more generally, persistent narratives valued in any category with limits; and (3) we establish a formal connection between these two perspectives using categorical duality. Specifically, we prove that there is an adjunction between the categories of persistent and cumulative narratives, providing a principled way to satisfy desideratum~\ref{desideratum-2}: sheaves encode the persistence perspective, while cosheaves—the dual of a sheaf—encode the accumulation perspective. As we will show, while these two perspectives give rise to equivalences on certain subcategories of temporal graphs, in general, when one passes to arbitrary categories of temporal objects—such as temporal groups, for example—this equivalence weakens to an \textit{adjunction} (this is Theorem~\ref{thm:adjunction}; roughly one can think of this as a Galois connection~\cite{seven_sketches}). In particular our results imply that in general there is the potential for a loss of information when one passes from one perspective (the persistent one, say) to another (the cumulative one) and back again. This observation, which has so far been ignored, is of great practical relevance since it means that one must take a great deal of care when collecting temporal data since the choices of mathematical representations may not be interchangeable. We will prove the existence of the adjunction between cumulative and persistent temporal data (for any category with pullbacks and pushouts, such as the category of graphs) in Theorem~\ref{thm:adjunction} and discuss all of these subtleties in Section~\ref{sec:cumulative-vs-persistent}. Furthermore, this adjunction opens interesting directions for future work investigating the relationship between the persistent and cumulative perspectives present in topological data analysis; for instance, the program of `generalized persistence' initiated by Patel and developed in the work of Kim and Memoli~\cite{Kim2023}.

\item Another common theme arising in temporal graph theory is the relationship between properties of static graphs and their temporal analogues. At first glance, one might na{\"i}vely think that static properties can be canonically lifted to the temporal setting by simply defining them in terms of underlying static graphs. However, this approach completely forgets the temporal structure and is thus of no use in generalizing notions such as for example connectivity or distance where temporal information is crucial to the intended application~\cite{OthonTempSurvey,FlocchiniQuattrocchiSantoroCasteigts,deleting_edge_temporal_graphs_Kitty_and_Jess,bumpus2022edge}.
Moreover, the lack of a systematic procedure for `\textit{temporalizing}' notions from static graph theory is more than an aesthetic obstacle. It fuels the proliferation of myriads of subtly different temporal analogues of static properties. For instance should a temporal coloring be a coloring of the underlying static graph?  
What about the underlying persistence graph? Or should it instead be a sequence of colorings? And should the colorings in this sequence be somehow related? 
Rather than accepting this proliferation as a mere consequence of the greater expressiveness of temporal data, we sublime these issues into desideratum~\ref{desideratum-3}: any theory of temporal data should come equipped with a systematic way of `temporalizing' notions from traditional, static mathematics. 

In Section \ref{sec:lifting-properties}, we show how our theories of narratives satisfies desideratum~\ref{desideratum-3}. We do so systematically by leveraging two simple, but effective \textit{functors}: the change of \textit{temporal resolution} functor (Proposition~\ref{prop:change-temp-resolution}) and the \textit{change of base} functor (Propositions~\ref{prop:change-of-base-covariant} and~\ref{prop:change-of-base-contravariant}). The first allows us to modify narratives by rescaling time, while the second allows us to change the kind of data involved in the narrative (e.g. passing from temporal simplicial complexes to temporal graphs). Using these tools, we provide a general way for temporalizing static notions which roughly allows one to start with a class of objects which satisfy a given property (e.g. the class of paths, if one is thinking about temporal graphs) and obtain from it a class of objects which temporally satisfy that property (e.g. the notion of temporal paths). As an example (other than temporal paths which we consider in Proposition~\ref{prop:def:temporal-path}) we apply our abstract machinery to recover in a canonical way (Proposition~\ref{prop:temporal-cliques}) the notion of a temporal clique (as defined by Viard, Latapy and Magnien~\cite{viard2016computing}). Crucially, the only information one needs to be given is the definition of a clique (in the static sense). Summarizing this last point with a slogan, one could say that `our formalism can derive the definition of temporal cliques given solely the definition of a clique as input'. Although it is beyond the scope of the present paper, we believe that this kind of reasoning will prove to be crucial in the future for a systematic study of how theories of temporal data (e.g. temporal graph theory) relate to their static counterparts (e.g. graph theory). 

\item Temporal graphs are definitely ubiquitous forms of temporal data~\cite{HARARY199779, OthonTempSurvey, HOLME201297, holme2015modern, FlocchiniQuattrocchiSantoroCasteigts, KEMPE2002820}, but they are by far not the only kind of temporal data one could attach, or sample from an underlying dynamical system. Thus Desideratum~\ref{desideratum-4} is evident: to further our understanding of data which changes with time, we cannot develop case by case theories of temporal graphs, temporal simplicial complexes, temporal groups etc., but instead we require a general theory of temporal data that encompasses all of these examples as specific instances and which allows us to relate different kinds of temporal data to each other. 

Our theory of narratives addresses part of Desideratum~\ref{desideratum-4} almost out of the box: our category theoretic formalism is object agnostic and can be thus applied to mathematical objects coming from any such category thereof. We observe through elementary constructions that there are \textit{change of base} functors which allow one to convert temporal data of one kind into temporal data of another. Furthermore, we observe that, when combined with the adjunction of Theorem~\ref{thm:adjunction}, these simple data conversions can rapidly lead to complex relationships between various kinds of temporal data. 

\item As we mentioned earlier, our philosophical contention is that on its own data is not temporal; it is through originating from an underlying dynamical system that its temporal nature is distilled. This link can and should be made explicit. But until now the development of such a general theory is impeded by a great mathematical and linguistic divide between the communities which study dynamics axiomatically (e.g. the study of differential equations, automata etc.) and those who study data (e.g. the study of time series, temporal graphs etc.). Thus we distill our last Desideratum~\ref{desideratum-5}: any theory of temporal data should be seamlessly interoperable with theories of dynamical systems from which the data can arise. 

This desideratum is ambitious enough to fuel a research program and is thus beyond the scope of a single paper. However, for any such theory to be developed, one first needs to place both the theory of dynamical systems and the theory of temporal data on the same mathematical and linguistic footing. This is precisely how our theory of narratives addresses Desideratum~\ref{desideratum-5}: since both narratives (our model of temporal data) and Schultz, Spivak and Vasilakopoulou's \textit{interval sheaves}~\cite{schultz2017temporal} (a general formalism for studying dynamical systems) are defined in terms of sheaves on categories of intervals, we have bridged a significant linguistic divide between the study of data and dynamics. We expect this to be a very fruitful line of further research in the years to come. 

\end{enumerate}

Having identified the desiderata for a theory of time-varying data, we conclude this section by further contextualizing our work with other uses of sheaf theory to study time-varying phenomena. As we already mentioned, Schultz, Spivak and Vasilakopoulou~\cite{schultz2017temporal} study dynamical systems through a sheaf-theoretic lens; moreover, there have been other investigations of time-varying structures which use similar tools. An example within the applied topology and topological data analysis communities is the examination of connected components over time using Reeb graphs. For instance, in~\cite{deSilva2016}, the authors leverage the established fact that the category of Reeb graphs is equivalent to a certain class of cosheaves. This equivalence is exploited to define a distance between Reeb graphs, which proves to be resilient to perturbations in the input data. Furthermore, it serves the purpose of smoothing the provided Reeb graphs in a manner that facilitates a geometric interpretation. Similarly, the study of the persistence of topological features in point-cloud datasets has given rise to the formulation of the theory of persistence for `Zigzag diagrams'. This theory extends beyond persistent homology and also has a cosheaf interpretation~\cite{Curry2015, curry2014sheaves}.  Although it is beyond the scope of the current paper, we believe that exploring the connections between our work these notions from applied topology is an exciting direction for further study.

\section{Categories of Temporal Data}\label{sec:cats-of-temporal-data}
Our thesis is that temporal data should be represented mathematically via \textit{sheaves} (or \textit{cosheaves}, their categorical dual). Sheaf theory, already established in the 1950s as a crucial tool in algebraic topology, complex analysis, and algebraic geometry, is canonically the study of local-to-global data management. For our purposes here, we will only make shallow use of this theory; nevertheless, we anticipate that more profound sheaf-theoretic tools, such as cohomology, will play a larger role in the future study of temporal data. To accommodate readers from disparate backgrounds, we will slowly build up the intuition for why one should represent temporal data as a sheaf by first peeking at examples of \textit{temporal sets} in Section~\ref{sec:temporal-sets}. We will then formally introduce interval sheaves (Section~\ref{sec:background}) and immediately apply them by collecting various examples of categories of temporal graphs (Section~\ref{sec:examples-of-narratives}) before ascending to more abstract theory.

\subsection{Garnering Intuition: Categories of Temporal Sets.}\label{sec:temporal-sets}
Take a city, like Venice, Italy, and envision documenting the set of ice cream companies that exist in that city each year. For instance, in the first year, there might be four companies \(\{a_1, a_2, b, c\}\). One could imagine that from the first year to the next, company \(b\) goes out of business, company \(c\) continues into the next year, a new ice cream company \(b'\) is opened, and the remaining two companies \(a_1\) and \(a_2\) merge into a larger company \(a_\star\). This is an example of a \textit{discrete temporal set} viewed from the perspective of \textit{persistence}: not only do we record the sets of companies each year (which, in the $i$-th year, we denote as $F_i^i$), but instead we also keep track of which companies persist from one year to the next and how they do so (which we denote as $F_i^j$). Diagramatically we could represent the first three years of this story as follows.
\begin{equation}\label{diagram:companies-example}
\adjustbox{scale=1.5, max width=.95\textwidth}{%,center}{

% https://q.uiver.app/#q=WzAsNSxbMSwwLCJGXzFeMiA9IFxce2FfMSwgYV8yLGNcXH0iXSxbMywwLCJGXzJeMyA9IFxce2FfXFxzdGFyLCBiJ1xcfSJdLFswLDEsIkZfMV4xOj1cXHthXzEsIGFfMiwgYiwgY1xcfSJdLFsyLDEsIkZfMl4yOj1cXHthX1xcc3RhcixiJyxjXFx9Il0sWzQsMSwiRl8zXjM6PVxce2FfXFxzdGFyLGInLCBjJ1xcfSJdLFswLDIsInt7Zl97MSwyfV4xfX0iLDAseyJzdHlsZSI6eyJ0YWlsIjp7Im5hbWUiOiJob29rIiwic2lkZSI6ImJvdHRvbSJ9fX1dLFswLDMsInt7Zl97MSwyfV4yfX0iLDJdLFsxLDMsInt7Zl97MiwzfV4yfX0iLDAseyJzdHlsZSI6eyJ0YWlsIjp7Im5hbWUiOiJob29rIiwic2lkZSI6ImJvdHRvbSJ9fX1dLFsxLDQsInt7Zl97MiwzfV4zfX0iLDIseyJzdHlsZSI6eyJ0YWlsIjp7Im5hbWUiOiJob29rIiwic2lkZSI6InRvcCJ9fX1dXQ==
\begin{tikzcd}
	& {F_1^2 = \{a_1, a_2,c\}} && {F_2^3 = \{a_\star, b'\}} \\
	{F_1^1:=\{a_1, a_2, b, c\}} && {F_2^2:=\{a_\star,b',c\}} && {F_3^3:=\{a_\star,b', c'\}}
	\arrow["{{{f_{1,2}^1}}}", hook', from=1-2, to=2-1]
	\arrow["{{{f_{1,2}^2}}}"', from=1-2, to=2-3]
	\arrow["{{{f_{2,3}^2}}}", hook', from=1-4, to=2-3]
	\arrow["{{{f_{2,3}^3}}}"', hook, from=1-4, to=2-5]
\end{tikzcd}
}
\end{equation}

This is a diagram of sets and the arrows are functions between sets. In this example we have that \(f_{1,2}^1\) is the canonical injection of \(F_1^2\) into \(F_1^1\) while \(f_{1,2}^2\) maps \(c\) to itself and it takes both \(a_1\) and \(a_2\) to \(a_\star\) (representing the unification of the companies \(a_1\) and \(a_2\)). 

Diagram~\ref{diagram:companies-example} is more than just a time-series or a sequence of sets: it tells a story by relating (via functions in this case) the elements of successive snapshots. It is obvious, however, that from the relationships shown in Diagram~\ref{diagram:companies-example} we should be able to recover longer-term relationships between instances in time. For instance we should be able to know what happened to the four companies \(\{a_1, a_2, b, c\}\) over the course of three years: by the third year we know that companies \(a_1\) and \(a_2\) unified and turned into company \(a_\star\), companies \(b\) and \(c\) dissolved and ceased to exist and two new companies \(b'\) and \(c'\) were born. 

The inferences we just made amounted to determining the relationship between the sets \(F_1^1\) and \(F_3^3\) completely from the data specified by Diagram~\ref{diagram:companies-example}. Mathematically this is an instance of computing \(F_1^3\) as a \textit{fibered product} (or \textit{pullback}) of the sets \(F_1^2\) and \(F_2^3\) along functions \(f_{1,2}^2\) and \(f_{2,3}^2\):  

\[F_1^3 := \{ (x,y) \in F_1^2 \times F_2^3 \mid f_{1,2}^2(x) = f_{2,3}^2(y)\}.\] Diagrammatically this is drawn as follows. 
\begin{equation}\label{diagram:companies-SHEAF-example}
% https://q.uiver.app/#q=WzAsNixbMCwyLCJGXzFeMTo9XFx7YV8xLCBhXzIsIGIsIGNcXH0iXSxbNCwyLCJGXzNeMzo9XFx7YSxiJywgYydcXH0iXSxbMSwxLCJGXzFeMiA9IFxce2FfMSwgYV8yLGNcXH0iXSxbMiwyLCJGXzJeMjo9XFx7YV9cXHN0YXIsYicsY1xcfSJdLFszLDEsIkZfMl4zID0gXFx7YV9cXHN0YXIsIGInXFx9Il0sWzIsMCwiRl8xXjM9XFx7KGFfMSwgYV9cXHN0YXIpLCAoYV8yLCBhX1xcc3RhcilcXH0iXSxbMiwwLCJmX3sxLDJ9XjEiLDAseyJzdHlsZSI6eyJ0YWlsIjp7Im5hbWUiOiJob29rIiwic2lkZSI6ImJvdHRvbSJ9fX1dLFsyLDMsImZfezEsMn1eMiIsMl0sWzQsMywiZl97MiwzfV4yIiwwLHsic3R5bGUiOnsidGFpbCI6eyJuYW1lIjoiaG9vayIsInNpZGUiOiJib3R0b20ifX19XSxbNCwxLCJmX3syLDN9XjMiLDAseyJzdHlsZSI6eyJ0YWlsIjp7Im5hbWUiOiJob29rIiwic2lkZSI6InRvcCJ9fX1dLFs1LDIsImZfezEsM31eezEsMn0iLDAseyJzdHlsZSI6eyJ0YWlsIjp7Im5hbWUiOiJob29rIiwic2lkZSI6ImJvdHRvbSJ9fX1dLFs1LDQsImZfezEsM31eezIsM30iLDJdLFs1LDMsIiIsMSx7InN0eWxlIjp7Im5hbWUiOiJjb3JuZXIifX1dXQ==
\adjustbox{scale=1.5, max width=.95\textwidth}{%,center}{
\begin{tikzcd}
	&& {F_1^3=\{(a_1, a_\star), (a_2, a_\star)\}} \\
	& {F_1^2 = \{a_1, a_2,c\}} && {F_2^3 = \{a_\star, b'\}} \\
	{F_1^1:=\{a_1, a_2, b, c\}} && {F_2^2:=\{a_\star,b',c\}} && {F_3^3:=\{a,b', c'\}}
	\arrow["{f_{1,3}^{1,2}}", hook', from=1-3, to=2-2]
	\arrow["{f_{1,3}^{2,3}}"', from=1-3, to=2-4]
	\arrow["\lrcorner"{anchor=center, pos=0.125, rotate=-45}, draw=none, from=1-3, to=3-3]
	\arrow["{f_{1,2}^1}", hook', from=2-2, to=3-1]
	\arrow["{f_{1,2}^2}"', from=2-2, to=3-3]
	\arrow["{f_{2,3}^2}", hook', from=2-4, to=3-3]
	\arrow["{f_{2,3}^3}", hook, from=2-4, to=3-5]
\end{tikzcd}
}
\end{equation}

The selection of the aforementioned data structures, namely sets and functions, allowed us to encode a portion of the history behind the ice cream companies in Venice. If we were to delve deeper and investigate, for instance, why company $c$ disappeared, we could explore a cause within the dynamics of the relationships between ice cream companies and their suppliers. These relationships  can be captured using directed graphs, as illustrated in Diagram \ref{diagram:companies-graph}, where there is an edge from $x$ to $y$ if the former is a supplier to the latter. This diagram reveals that company $a_2$ not only sold ice cream but also supplied companies $a_1$ and $c$.  Notably, with the dissolution of company $b$ in the second year, it becomes conceivable that the closure of company $c$ occurred due to the cessation of its supply source.

% Diagram: graph example
%\input{diagrams/diagram2}
%\begin{equation}\label{diagram:companies-graph}
    %\includegraphics[width=0.9\linewidth]{diagrams/diagram-graph-narrative.pdf}
    \begin{equation}\label{diagram:companies-graph}
% https://q.uiver.app/#q=WzAsNixbMCwyLCJGXzFeMTo9XFx7YV8xLCBhXzIsIGIsIGNcXH0iXSxbNCwyLCJGXzNeMzo9XFx7YSxiJywgYydcXH0iXSxbMSwxLCJGXzFeMiA9IFxce2FfMSwgYV8yLGNcXH0iXSxbMiwyLCJGXzJeMjo9XFx7YV9cXHN0YXIsYicsY1xcfSJdLFszLDEsIkZfMl4zID0gXFx7YV9cXHN0YXIsIGInXFx9Il0sWzIsMCwiRl8xXjM9XFx7KGFfMSwgYV9cXHN0YXIpLCAoYV8yLCBhX1xcc3RhcilcXH0iXSxbMiwwLCJmX3sxLDJ9XjEiLDAseyJzdHlsZSI6eyJ0YWlsIjp7Im5hbWUiOiJob29rIiwic2lkZSI6ImJvdHRvbSJ9fX1dLFsyLDMsImZfezEsMn1eMiIsMl0sWzQsMywiZl97MiwzfV4yIiwwLHsic3R5bGUiOnsidGFpbCI6eyJuYW1lIjoiaG9vayIsInNpZGUiOiJib3R0b20ifX19XSxbNCwxLCJmX3syLDN9XjMiLDAseyJzdHlsZSI6eyJ0YWlsIjp7Im5hbWUiOiJob29rIiwic2lkZSI6InRvcCJ9fX1dLFs1LDIsImZfezEsM31eezEsMn0iLDAseyJzdHlsZSI6eyJ0YWlsIjp7Im5hbWUiOiJob29rIiwic2lkZSI6ImJvdHRvbSJ9fX1dLFs1LDQsImZfezEsM31eezIsM30iLDJdLFs1LDMsIiIsMSx7InN0eWxlIjp7Im5hbWUiOiJjb3JuZXIifX1dXQ==
\adjustbox{scale=1.5, max width=.95\textwidth}{%,center}{
\tikzset{every picture/.style={line width=0.75pt}} %set default line width to 0.75pt 
\begin{tikzpicture}[x=0.75pt,y=0.75pt,yscale=-1,xscale=1]
%uncomment if require: \path (0,165); %set diagram left start at 0, and has height of 165

%Rounded Rect [id:dp651255313932755] 
\draw  [color={rgb, 255:red, 255; green, 255; blue, 255 }  ,draw opacity=1 ][fill={rgb, 255:red, 255; green, 192; blue, 203 }  ,fill opacity=1 ] (593.36,115.84) .. controls (593.36,109.8) and (598.25,104.91) .. (604.29,104.91) -- (688.28,104.91) .. controls (694.32,104.91) and (699.21,109.8) .. (699.21,115.84) -- (699.21,148.62) .. controls (699.21,154.66) and (694.32,159.55) .. (688.28,159.55) -- (604.29,159.55) .. controls (598.25,159.55) and (593.36,154.66) .. (593.36,148.62) -- cycle ;
%Rounded Rect [id:dp964021196355761] 
\draw  [color={rgb, 255:red, 255; green, 255; blue, 255 }  ,draw opacity=1 ][fill={rgb, 255:red, 255; green, 192; blue, 203 }  ,fill opacity=1 ] (307.02,112.5) .. controls (307.02,106.46) and (311.91,101.57) .. (317.95,101.57) -- (428.42,101.57) .. controls (434.46,101.57) and (439.35,106.46) .. (439.35,112.5) -- (439.35,145.28) .. controls (439.35,151.32) and (434.46,156.21) .. (428.42,156.21) -- (317.95,156.21) .. controls (311.91,156.21) and (307.02,151.32) .. (307.02,145.28) -- cycle ;
%Rounded Rect [id:dp49330610581151313] 
\draw  [color={rgb, 255:red, 255; green, 255; blue, 255 }  ,draw opacity=1 ][fill={rgb, 255:red, 255; green, 192; blue, 203 }  ,fill opacity=1 ] (197.46,16.66) .. controls (197.46,10.48) and (202.47,5.46) .. (208.66,5.46) -- (248.18,5.46) .. controls (254.36,5.46) and (259.37,10.48) .. (259.37,16.66) -- (259.37,50.23) .. controls (259.37,56.41) and (254.36,61.42) .. (248.18,61.42) -- (208.66,61.42) .. controls (202.47,61.42) and (197.46,56.41) .. (197.46,50.23) -- cycle ;
%Rounded Rect [id:dp03504872059171682] 
\draw  [color={rgb, 255:red, 255; green, 255; blue, 255 }  ,draw opacity=1 ][fill={rgb, 255:red, 255; green, 192; blue, 203 }  ,fill opacity=1 ] (27.45,112.46) .. controls (27.45,106.42) and (32.34,101.53) .. (38.38,101.53) -- (146.1,101.53) .. controls (152.13,101.53) and (157.03,106.42) .. (157.03,112.46) -- (157.03,145.24) .. controls (157.03,151.28) and (152.13,156.17) .. (146.1,156.17) -- (38.38,156.17) .. controls (32.34,156.17) and (27.45,151.28) .. (27.45,145.24) -- cycle ;
%Straight Lines [id:da6376299448583339] 
\draw    (73.95,116.55) -- (61.06,126.41) -- (49.64,135.16) ;
\draw [shift={(47.26,136.98)}, rotate = 322.56] [fill={rgb, 255:red, 0; green, 0; blue, 0 }  ][line width=0.08]  [draw opacity=0] (10.72,-5.15) -- (0,0) -- (10.72,5.15) -- (7.12,0) -- cycle    ;
%Straight Lines [id:da9655825520807056] 
\draw    (88.73,116.72) -- (106.73,135.68) ;
\draw [shift={(108.79,137.85)}, rotate = 226.49] [fill={rgb, 255:red, 0; green, 0; blue, 0 }  ][line width=0.08]  [draw opacity=0] (10.72,-5.15) -- (0,0) -- (10.72,5.15) -- (7.12,0) -- cycle    ;
%Straight Lines [id:da595887692491491] 
\draw    (142.11,117.42) -- (122.83,136.44) ;
\draw [shift={(120.69,138.55)}, rotate = 315.39] [fill={rgb, 255:red, 0; green, 0; blue, 0 }  ][line width=0.08]  [draw opacity=0] (10.72,-5.15) -- (0,0) -- (10.72,5.15) -- (7.12,0) -- cycle    ;
%Straight Lines [id:da8282179015755521] 
\draw    (222.36,24.86) -- (238.2,40.93) ;
\draw [shift={(240.3,43.07)}, rotate = 225.42] [fill={rgb, 255:red, 0; green, 0; blue, 0 }  ][line width=0.08]  [draw opacity=0] (10.72,-5.15) -- (0,0) -- (10.72,5.15) -- (7.12,0) -- cycle    ;
%Rounded Rect [id:dp970602085460684] 
\draw  [color={rgb, 255:red, 255; green, 255; blue, 255 }  ,draw opacity=1 ][fill={rgb, 255:red, 255; green, 192; blue, 203 }  ,fill opacity=1 ] (484.58,16.31) .. controls (484.58,10.3) and (489.45,5.42) .. (495.46,5.42) -- (535.6,5.42) .. controls (541.62,5.42) and (546.49,10.3) .. (546.49,16.31) -- (546.49,48.96) .. controls (546.49,54.97) and (541.62,59.84) .. (535.6,59.84) -- (495.46,59.84) .. controls (489.45,59.84) and (484.58,54.97) .. (484.58,48.96) -- cycle ;
%Straight Lines [id:da01807202060692248] 
\draw    (501.41,22.21) -- (524.81,41.3) ;
\draw [shift={(527.14,43.2)}, rotate = 219.2] [fill={rgb, 255:red, 0; green, 0; blue, 0 }  ][line width=0.08]  [draw opacity=0] (10.72,-5.15) -- (0,0) -- (10.72,5.15) -- (7.12,0) -- cycle    ;
%Straight Lines [id:da7056933848921817] 
\draw    (477.72,64.97) -- (462.53,80.61) -- (449.61,93.22) ;
\draw [shift={(448.18,94.62)}, rotate = 315.7] [color={rgb, 255:red, 0; green, 0; blue, 0 }  ][line width=0.75]    (10.93,-4.9) .. controls (6.95,-2.3) and (3.31,-0.67) .. (0,0) .. controls (3.31,0.67) and (6.95,2.3) .. (10.93,4.9)   ;
%Straight Lines [id:da778897724532588] 
\draw    (353.17,116.87) -- (341.55,125.93) -- (333.31,132.36) ;
\draw [shift={(330.94,134.2)}, rotate = 322.05] [fill={rgb, 255:red, 0; green, 0; blue, 0 }  ][line width=0.08]  [draw opacity=0] (10.72,-5.15) -- (0,0) -- (10.72,5.15) -- (7.12,0) -- cycle    ;
%Straight Lines [id:da8208788756770781] 
\draw    (640.04,120.07) -- (615.18,138.7) ;
\draw [shift={(612.78,140.5)}, rotate = 323.14] [fill={rgb, 255:red, 0; green, 0; blue, 0 }  ][line width=0.08]  [draw opacity=0] (10.72,-5.15) -- (0,0) -- (10.72,5.15) -- (7.12,0) -- cycle    ;
%Straight Lines [id:da5099089036962001] 
\draw    (655.14,120.25) -- (673.54,139.22) ;
\draw [shift={(675.62,141.38)}, rotate = 225.89] [fill={rgb, 255:red, 0; green, 0; blue, 0 }  ][line width=0.08]  [draw opacity=0] (10.72,-5.15) -- (0,0) -- (10.72,5.15) -- (7.12,0) -- cycle    ;
%Straight Lines [id:da04794267732743929] 
\draw    (187.53,64.72) -- (160.19,92.54) ;
\draw [shift={(158.78,93.97)}, rotate = 314.51] [color={rgb, 255:red, 0; green, 0; blue, 0 }  ][line width=0.75]    (10.93,-4.9) .. controls (6.95,-2.3) and (3.31,-0.67) .. (0,0) .. controls (3.31,0.67) and (6.95,2.3) .. (10.93,4.9)   ;
%Straight Lines [id:da9256611021932947] 
\draw    (267.78,64.47) -- (295.89,93.53) ;
\draw [shift={(297.28,94.97)}, rotate = 225.95] [color={rgb, 255:red, 0; green, 0; blue, 0 }  ][line width=0.75]    (10.93,-4.9) .. controls (6.95,-2.3) and (3.31,-0.67) .. (0,0) .. controls (3.31,0.67) and (6.95,2.3) .. (10.93,4.9)   ;
%Straight Lines [id:da33396431136558724] 
\draw    (557.97,65.22) -- (586.31,93.56) ;
\draw [shift={(587.72,94.97)}, rotate = 225] [color={rgb, 255:red, 0; green, 0; blue, 0 }  ][line width=0.75]    (10.93,-4.9) .. controls (6.95,-2.3) and (3.31,-0.67) .. (0,0) .. controls (3.31,0.67) and (6.95,2.3) .. (10.93,4.9)   ;
%Curve Lines [id:da7164314497047067] 
\draw    (361.32,119.01) .. controls (365.9,136.59) and (383.96,128.46) .. (367.28,114.89) ;
\draw [shift={(365.01,113.17)}, rotate = 35.44] [fill={rgb, 255:red, 0; green, 0; blue, 0 }  ][line width=0.08]  [draw opacity=0] (10.72,-5.15) -- (0,0) -- (10.72,5.15) -- (7.12,0) -- cycle    ;
%Curve Lines [id:da3765800091890339] 
\draw    (493.68,25.97) .. controls (471.49,40.85) and (514.55,60.48) .. (499.88,27.16) ;
\draw [shift={(498.63,24.49)}, rotate = 63.85] [fill={rgb, 255:red, 0; green, 0; blue, 0 }  ][line width=0.08]  [draw opacity=0] (10.72,-5.15) -- (0,0) -- (10.72,5.15) -- (7.12,0) -- cycle    ;
%Curve Lines [id:da6289701202158322] 
\draw    (644.68,122.3) .. controls (622.49,137.19) and (665.55,156.82) .. (650.88,123.5) ;
\draw [shift={(649.63,120.83)}, rotate = 63.85] [fill={rgb, 255:red, 0; green, 0; blue, 0 }  ][line width=0.08]  [draw opacity=0] (10.72,-5.15) -- (0,0) -- (10.72,5.15) -- (7.12,0) -- cycle    ;

% Text Node
\draw (1.67,115.51) node [anchor=north west][inner sep=0.75pt]    {$F_{1}^{1}$};
% Text Node
\draw (277.62,116.7) node [anchor=north west][inner sep=0.75pt]    {$F_{2}^{2}$};
% Text Node
\draw (242.66,39.39) node [anchor=north west][inner sep=0.75pt]   [align=left] {$\displaystyle a_{1}$};
% Text Node
\draw (207.32,5.76) node [anchor=north west][inner sep=0.75pt]   [align=left] {$\displaystyle a_{2}$};
% Text Node
\draw (243.14,8.18) node [anchor=north west][inner sep=0.75pt]   [align=left] {$\displaystyle c$};
% Text Node
\draw (172.02,22.51) node [anchor=north west][inner sep=0.75pt]    {$F_{1}^{2}$};
% Text Node
\draw (526.27,40.17) node [anchor=north west][inner sep=0.75pt]   [align=left] {$\displaystyle b'$};
% Text Node
\draw (490.94,6.55) node [anchor=north west][inner sep=0.75pt]   [align=left] {$\displaystyle a^{*}$};
% Text Node
\draw (454.26,21.73) node [anchor=north west][inner sep=0.75pt]    {$F_{2}^{3}$};
% Text Node
\draw (561.82,123.8) node [anchor=north west][inner sep=0.75pt]    {$F_{3}^{3}$};
% Text Node
\draw (314.46,132.6) node [anchor=north west][inner sep=0.75pt]   [align=left] {$\displaystyle b'$};
% Text Node
\draw (354.8,99.43) node [anchor=north west][inner sep=0.75pt]   [align=left] {$\displaystyle a^{*}$};
% Text Node
\draw (392.31,134.46) node [anchor=north west][inner sep=0.75pt]   [align=left] {$ $};
% Text Node
\draw (398.1,102.15) node [anchor=north west][inner sep=0.75pt]   [align=left] {$\displaystyle c$};
% Text Node
\draw (597.99,134.27) node [anchor=north west][inner sep=0.75pt]   [align=left] {$\displaystyle b'$};
% Text Node
\draw (642.33,102.92) node [anchor=north west][inner sep=0.75pt]   [align=left] {$\displaystyle a^{*}$};
% Text Node
\draw (676.84,134.61) node [anchor=north west][inner sep=0.75pt]   [align=left] {$\displaystyle c'$};
% Text Node
\draw (472.28,82.03) node [anchor=north west][inner sep=0.75pt]   [align=left] {$\displaystyle f_{2,3}^{2}$};
% Text Node
\draw (182.78,82.53) node [anchor=north west][inner sep=0.75pt]   [align=left] {$\displaystyle f_{1,2}^{1}$};
% Text Node
\draw (251.28,81.53) node [anchor=north west][inner sep=0.75pt]   [align=left] {$\displaystyle f_{1,2}^{2}$};
% Text Node
\draw (540.78,82.53) node [anchor=north west][inner sep=0.75pt]   [align=left] {$\displaystyle f_{2,3}^{3}$};
% Text Node
\draw (34.56,132.56) node [anchor=north west][inner sep=0.75pt]   [align=left] {$\displaystyle a_{1}$};
% Text Node
\draw (73.57,99.39) node [anchor=north west][inner sep=0.75pt]   [align=left] {$\displaystyle a_{2}$};
% Text Node
\draw (110.85,134.42) node [anchor=north west][inner sep=0.75pt]   [align=left] {$\displaystyle c$};
% Text Node
\draw (142.47,102.12) node [anchor=north west][inner sep=0.75pt]   [align=left] {$\displaystyle b$};
\end{tikzpicture}
}
\end{equation}     
%\end{equation}

More generally, within a system, numerous observations can be made. Each observation is intended to capture a different facet of the problem. This diversity translates into the necessity of employing various data structures, such as sets, graphs, groups, among others, to represent relevant mathematical spaces underlying the data. Our goal in this work is to use a language that enables us to formally handle data whose snapshots are modeled via commonly used data structures in data analysis.  As we will explain in Section~\ref{sec:background}, the language we are looking for is that of \textit{sheaves}, and the structure hidden in Diagrams~\ref{diagram:companies-SHEAF-example} and \ref{diagram:companies-graph} is that of a \textit{sheaf on a category of intervals}. Sheaves are most naturally described in category-theoretic terms and, as is always the case in category theory, they admit a categorically dual notion, namely \textit{cosheaves}. As it turns out, while sheaves capture the notion of \textit{persistent objects}, cosheaves on interval categories instead capture the idea of an underlying static object that is \textit{accumulated over time}. Thus we see (this will be explained formally in Section~\ref{sec:cumulative-vs-persistent}) that the two perspectives—\textit{persistent} vs. \textit{cumulative}—of our second desideratum are not only convenient and intuitively natural, they are also dual to each other in a formal sense. 

\subsection{Narratives}\label{sec:background}
From this section onward we will assume basic familiarity with categories, functors and natural transformations. For a very short, self-contained introduction to the necessary background suitable for graph theorists, we refer the reader to the thesis by Bumpus~\cite[Sec. 3.2]{bumpus2021generalizing}. For a thorough introduction to the necessary category-theoretic background, we refer the reader to any monograph on category theory (such as Riehl's textbook~\cite{riehl2017category} or Awodey's~\cite{awodey2010category}). We will give concrete definitions of the specific kinds of sheaves and co-sheaves that feature in this paper; however, we shall not recall standard notions in sheaf theory. For an approachable introduction to any notion from sheaf theory not explicitly defined here, we refer the reader to Rosiak's excellent textbook~\cite{rosiak-book}.

For most, the first sheaves one encounters are sheaves on a topological space. These are assignments of data to each open of a given topological space in such a way that these data can be restricted along inclusions of opens and such that the data assigned to any open \(U\) of the space is completely determined from the data assigned to the opens of any \textit{cover} of \(U\). In gradually more concrete terms, a \textit{\(\cat{Set}\)-valued sheaf} \(\Fa\) on a topological space \(\Xa\) is a contravariant functor (a \textit{presheaf}) \(\Fa \colon \Oa(\Xa)^{op} \to \cat{Set}\) from the poset of opens in \(\Xa\) to sets which satisfies certain lifting properties relating the values of \(\Fa\) on any open \(U\) to the values of \((\Fa(U_i))_{i \in I}\) for any open cover \((U_i)_{i \in I}\) of \(U\). Here we are interested in sheaves that are: (1) defined on posets (categories) of closed intervals of the non-negative reals (or integers) and (2) not necessarily \(\cat{Set}\)-valued. The first requirement has to do with representing time. Each point in time \(t\) is represented by a singleton interval \([t,t]\) and each proper interval \([t_1,t_2]\) accounts for the time spanned between its endpoints. The second requirement has to do with the fact that we are not merely interested in temporal sets, but instead we wish to build a more general theory capable or representing with a single formalism many kinds of temporal data such as temporal graphs, temporal topological spaces, temporal databases, temporal groups etc.. 

Thus one can see that, in order to specify a sheaf, one requires: (1) a presheaf \(\Fa \colon \cat{C}^{op} \to \cat{D}\) from a category \(\cat{C}\) to a category \(\cat{D}\), (2) a notion of what should count of as a `cover' of any object of \(\cat{C}\) and (3) a formalization of how \(\Fa\) should relate objects to their covers. To address the first point we will first give a reminder of the more general notation and terminology surrounding presheaves. 

\begin{definition}\label{def:presheaf}
For any small category \(\cat{C}\) (such as \(\cat{I}\) or \(\cat{I}_\mathbb{N}\)) we denote by \(\cat{D}^{\cat{C}}\) the category of \define{\(\cat{D}\)-valued co-presheaves on \(\cat{C}\)}; this has functors \(P \colon \cat{C} \to \cat{D}\) as objects and natural transformations as morphisms. When we wish to emphasize contravariance, we call \(\cat{D}^{\cat{C}^{op}}\) the \define{category of \(\cat{D}\)-valued presheaves on \(\cat{C}\)}.
\end{definition}

The second point -- on choosing good notions of `covers' -- is smoothly handled via the notion of a \define{Grothendieck topology} (see Rosiak's textbook~\cite{rosiak-book} for a formal definition). Categories equipped with a choice of a Grothendieck topology are known as \define{sites} and the following definition (due to Schultz, Spivak and Vasilakopoulou~\cite{schultz2017temporal}) amounts to a way of turning categories of intervals into sites by specifying what counts as a valid cover of any interval.   

\begin{definition}[Interval categories~\cite{schultz2020dynamical}]\label{def:interval}
The \define{category of intervals}, denoted \(\cat{Int}\) is the category having closed intervals \([\ell, \ell']\) in \(\mathbb{R}+\) (the non-negative reals) as objects and orientation-preserving isometries as morphisms. Analogously, one can define the category \(\cat{Int}_\mathbb{N}\) of \define{discrete intervals} by restricting only to \(\mathbb{N}\)-valued intervals. These categories can be turned into sites by equipping them with the Johnstone coverage~\cite{schultz2020dynamical} which stipulates that any \define{cover} of any interval \([\ell, \ell']\) is generated by a partition into two closed intervals \(([\ell, p], [p, \ell'])\).\footnote{Thus note that, for any interval $[a,b]$, a valid cover can be specified by simply choosing any number of distinct points $p_1, \dots, p_n$ such that $a \leq p_1 \leq \dots \leq p_n \leq b$; this yields the cover $\{[a,p_1], [p_1, p_2], \dots [p_{n-1}, p_n], [p_n, b]\}$.}
\end{definition}

Schultz, Spivak and Vasilakopoulou defined interval sites to speak of \textit{dynamical systems} as sheaves~\cite{schultz2020dynamical}. Here we are instead interested in temporal data. As most would expect, data should in general be less temporally interwoven compared to its dynamical system of provenance (after all the temporal data should carry less information than a dynamical system). This intuition\footnote{By comparing examples of interval sheaves with sheaves on categories of strict intervals, the reader can verify that there is a sense in which these intuitions can be made mathematically concrete (in order to not derail the presentation of this paper, we omit these examples).} motivates why we will not work directly with Schultz, Spivak and Vasilakopoulou's definition, but instead we will make use of the following stricter notion of \textit{categories of strict intervals}.\footnote{Note that there is a sense in which a functor defined on a subcategory of some category \(\cat{C}\) has greater freedom compared to a functor defined on all of \(\cat{C}\). This is because there are fewer arrows (and hence fewer \textit{equations}) which need to be accounted for in the subcategory.}  

\begin{definition}[Closed Intervals and Inclusions]\label{def:strict-interval}
We denote by \(\cat{I}\) (resp. \(\cat{I}_\mathbb{N}\)) the full subcategory (specifically a join-semilattice) of the subobject poset of $\mathbb{R}$ (resp. $\mathbb{N}$) whose objects are closed intervals and whose morphisms are inclusions of intervals.  
\end{definition}

Clearly, the categories defined above are subcategories of $\cat{Int}$ (resp. $\cat{Int}_{\mathbb{N}}$) since their morphisms are orientation-preserving isometries. Notice that the categories \(\cat{I}\) (resp. \(\cat{I}_\mathbb{N}\)) are posetal (in contrast, \(\cat{Int}\) is \textit{not} posetal).  Hence, the poset of subobjects of any interval \([a,b]\) forms a subcategory of \(\cat{I}\) (resp \(\cat{I}_\mathbb{N}\)), which we denote by \(\cat{I}/[a,b]\) (resp. \(\cat{I}_{\mathbb{N}}/[a,b]\)). In what follows, since we will want to speak of discrete, continuous, finite and infinite time, it will be convenient to have terminology to account for which categories we will allow as models of time. We will call such categories \textit{time categories}. 

\begin{notation}
A \define{time category} is any sub-join-semilattice \(\cat{T}\) of either \(\cat{I}\) or \(\cat{I}_\mathbb{N}\). The left hand side of Figure~\ref{fig:schematic-drawing-of-a-narrative} visualizes the time category \(\cat{I}_{\mathbb{N}}/[1,3]\).
\end{notation}

The following lemma states that time categories can be given Grothendieck topologies in much the same way as the interval categories of Definition~\ref{def:interval}. Since the proof is completely routine, but far too technical for newcomers to sheaf theory, we will omit it assuming that the readers well-versed in sheaf theory can reproduce it on their own.

\begin{lemma}\label{lemma:time-category-sites}
Any time category forms a site when equipped with the Johnstone coverage. 
\end{lemma}

Equipped with suitable sites, we are now ready to give the definition of the categories \(\cat{Cu}(\cat{T}, \cat{D})\) and \(\cat{Pe}(\cat{T}, \cat{D})\) where \(\cat{T}\) is any time category. We will refer to either one of these as categories of \define{\(\cat{D}\)-narratives in \(\cat{T}\)-time}: intuitively these are categories whose objects are time-varying objects of \(\cat{D}\). For instance, taking \(\cat{D}\) to be \(\cat{Set}\) or \(\cat{Grph}\) one can speak of time varying sets or time-varying graphs. The difference between \(\cat{Pe}(\cat{T}, \cat{D})\) and \(\cat{Cu}(\cat{T}, \cat{D})\) will be that the first encodes \(\cat{D}\)-narratives according to the \textit{persistent} perspective (these will be \(\cat{D}\)-valued \textit{sheaves} on \(\cat{T}\)), while the second employs a \textit{cumulative} one (these will be \(\cat{D}\)-valued \textit{co-}sheaves on \(\cat{T}\)). 

\begin{definition}\label{def:discrete-and-lifetime}
We will say that the narratives are \define{discrete} if the time category involved is either \(\cat{I}_\mathbb{N}\) or any sub-join-semilattices thereof. Similarly we will say that a category of narratives has \define{finite lifetime} if its time category has finitely many objects or if it is a subobject poset generated by some element of \(\cat{I}\) or \(\cat{I}_\mathbb{N}\).    
\end{definition}

Now we are ready to give the definition of a sheaf with respect to any of the sites described in Lemma~\ref{lemma:time-category-sites}. The reader not interested in sheaf theory should take the following proposition (whose proof is a mere instantiation of the standard definition of a sheaf on a site) as a \textit{definition} of a sheaf on a time category.  

\begin{figure}
    \centering
    \includegraphics[width=\textwidth]{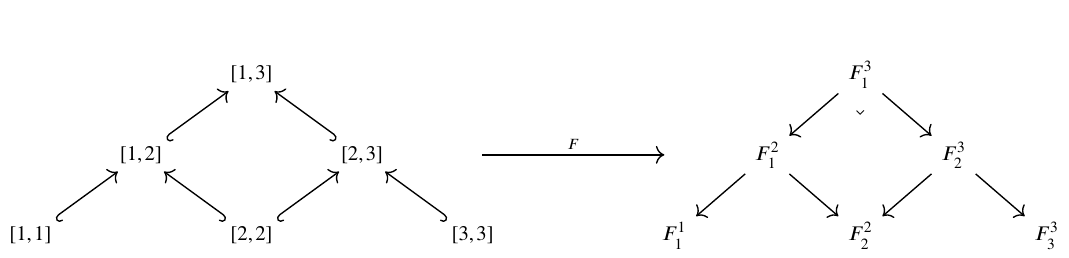}
    \caption{A schematic visualization of a sheaf on a discrete time category (a persistent narrative) with three snapshots. The domain is the time category \(\cat{I}_{\mathbb{N}}/[1,3]\), a join-semi-lattice whose objects are closed subintervals of \([1,3]\) and whose morphisms are interval inclusions. The codomain of the sheaf is a category \(\cat{D}\) with pullbacks. We use the shorthand $F_i^j$ (for $i \leq j$) to denote the data assigned the interval $[i,j]$ by $F$ (i.e. $F_i^j := F([i,j])$).}
    \label{fig:schematic-drawing-of-a-narrative}
\end{figure}

\begin{proposition}[\(\cat{T}\)-sheaves and \(\cat{T}\)-cosheaves]\label{prop:def:sheaves}
Let \(\cat{T}\) be any time category equipped with the Johnstone coverage. Suppose \(\cat{D}\) is a category with pullbacks, then a \define{\(\cat{D}\)-valued sheaf on \(\cat{T}\)} is a presheaf \(F \colon \cat{T}^{op} \to \cat{D}\) satisfying the following additional condition: for any interval \([a,b]\) and any cover \(([a, p],[p, b])\) of this interval, \(F([a,b])\) is the pullback \(F([a,p]) \times_{F([p,p])} F([p,b])\). (See Figure~\ref{fig:schematic-drawing-of-a-narrative}.)

Similarly, supposing \(\cat{D}\) to be a category with pushouts, then a \define{\(\cat{D}\)-valued cosheaf on \(\cat{T}\)} is a copresheaf \(\hat{F} \colon \cat{T} \to \cat{D}\) satisfying the following additional condition: for any interval \([a,b]\) and any cover \(([a, p],[p, b])\) of this interval, \(\hat{F}([a,b])\) is the pushout \(\hat{F}([a,p]) +_{\hat{F}([p,p])} \hat{F}([p,b])\).
\end{proposition}
\begin{proof}
By definition, a sheaf (resp. cosheaf) on the Johnstone coverage is simply a presheaf which takes each cover (a partion of an interval) to a limit (resp. colimit). 
\end{proof}

\begin{definition}\label{def:narratives}
We denote by \(\cat{Pe}(\cat{T}, \cat{D})\) (resp. \(\cat{Cu}(\cat{T}, \cat{D})\)) the category of \(\cat{D}\)-valued sheaves (resp. cosheaves) on \(\cat{T}\) and we call it the category of \define{persistent \(\cat{D}\)-narratives} (resp. \define{cumulative \(\cat{D}\)-narratives}) with \(\cat{T}\)-time.
\end{definition}

By this point, the reader has already encountered an example of a \textit{persistent discrete \(\cat{Set}\)-narrative}: Diagram~\ref{diagram:companies-SHEAF-example}, which illustrates the evolution of the temporal set over only three time steps. In contrast, Diagram~\ref{diagram:not-a-persistent-narrative} is \textit{not} a persistent \(\cat{Set}\)-narrative since \(F_1^3 \neq F_1^2 \times_{F_2^2} F_2^3\). To see this, observe that \(F_1^2 \times_{F_2^2} F_2^3\) is the pullback of two subsets (indicated by the hooked arrows denoting injective maps), each of size two. Thus, \(F_1^2 \times_{F_2^2} F_2^3\) has at most four elements, whereas \(F_1^3\) has five.

\begin{equation}\label{diagram:not-a-persistent-narrative}
% https://q.uiver.app/#q=WzAsNixbMCwyLCJGXzFeMTo9XFx7YV8xLCBhXzIsIGIsIGNcXH0iXSxbNCwyLCJGXzNeMzo9XFx7YSxiJywgYydcXH0iXSxbMSwxLCJGXzFeMiA9IFxce2FfMSwgYV8yLGNcXH0iXSxbMiwyLCJGXzJeMjo9XFx7YV9cXHN0YXIsYicsY1xcfSJdLFszLDEsIkZfMl4zID0gXFx7YV9cXHN0YXIsIGInXFx9Il0sWzIsMCwiRl8xXjM9XFx7KGFfMSwgYV9cXHN0YXIpLCAoYV8yLCBhX1xcc3RhcilcXH0iXSxbMiwwLCJmX3sxLDJ9XjEiLDAseyJzdHlsZSI6eyJ0YWlsIjp7Im5hbWUiOiJob29rIiwic2lkZSI6ImJvdHRvbSJ9fX1dLFsyLDMsImZfezEsMn1eMiIsMl0sWzQsMywiZl97MiwzfV4yIiwwLHsic3R5bGUiOnsidGFpbCI6eyJuYW1lIjoiaG9vayIsInNpZGUiOiJib3R0b20ifX19XSxbNCwxLCJmX3syLDN9XjMiLDAseyJzdHlsZSI6eyJ0YWlsIjp7Im5hbWUiOiJob29rIiwic2lkZSI6InRvcCJ9fX1dLFs1LDIsImZfezEsM31eezEsMn0iLDAseyJzdHlsZSI6eyJ0YWlsIjp7Im5hbWUiOiJob29rIiwic2lkZSI6ImJvdHRvbSJ9fX1dLFs1LDQsImZfezEsM31eezIsM30iLDJdLFs1LDMsIiIsMSx7InN0eWxlIjp7Im5hbWUiOiJjb3JuZXIifX1dXQ==
\adjustbox{scale=1.5, max width=.95\textwidth}{%,center}{
\begin{tikzcd}
	&& {F_1^3=\{a,w,x,y,z\}} \\
	& {F_1^2 = \{a,c\}} && {F_2^3 = \{a, b'\}} \\
	{F_1^1:=\{a,b,c\}} && {F_2^2:=\{a,b',c\}} && {F_3^3:=\{a,b', c'\}}
	\arrow[from=1-3, to=2-2]
	\arrow[from=1-3, to=2-4]
	\arrow[hook', from=2-2, to=3-1]
	\arrow[hook, from=2-2, to=3-3]
	\arrow[hook', from=2-4, to=3-3]
	\arrow[hook, from=2-4, to=3-5]
\end{tikzcd}
}
\end{equation}

When writing examples, it is useful to observe that all \textit{discrete} \(\cat{D}\)-narratives (see Definition~\ref{def:discrete-and-lifetime}) are completely determined by the objects and morphisms associated to intervals of length \textit{zero} and \textit{one}. This also implies, for example, that, in order to store a discrete graph narrative with $t$-time steps, it suffices to store $2t - 1$ graphs (one for each interval of length zero and one for each interval of length one) and $2(t-1)$ graph homomorphisms.

\begin{proposition}\label{prop:specifying-a-discrete-narrative}
Suppose we are given a objects \(F([t,t])\) and \(F([t, t+1])\) of \(\cat{D}\) for each time point \([t,t]\) and for each length-one interval \([t, t+1]\) and that we are furthermore given a span \(F([t, t]) \leftarrow F([t, t+1]) \rightarrow  F([t+1, t+1])\) for each pair of successive times $t$ and $t+1$. Then there is (up to isomorphism) a \textit{unique} discrete \(\cat{D}\)-narrative which agrees with these choices of objects and spans. Conversely, a mere sequence of objects of $\cat{D}$ (i.e. a choice of one object for each interval of length zero) does not determine a unique discrete $\cat{D}$-narrative. 
\end{proposition}
\begin{proof}
To see the first point, simply observe that applying the sheaf condition to this data leaves no choice for the remaining assignments on objects and arrows: these are completely determined by pullback and pullbacks are unique up to isomorphism.

On the other hand, suppose we are only given a list of objects of $\cat{D}$, one for each interval of length zero. Then, having to satisfy the sheaf condition does \textit{not} determine a unique $\cat{D}$-narrative that agrees with the given snapshots. To see this, observe that any length-one interval \([t,t+1]\) has exactly two covers: namely \(([t,t], [t, t+1])\) and \(([t,t + 1], [t + 1, t+1])\). Thus, applying the sheaf condition, we we have that \(G([t, t+1])\) must be the pullback \(G([t,t]) \times_{G([t,t])} G([t,t+1])\). However, this pullback is \textit{always} isomorphic to \(G([t,t+1])\) for any choice of the object \(G([t,t+1])\) since pullbacks preserve isomorphisms (and since the restriction of \(G([t,t])\) to itself is its identity morphism). This concludes the proof since the argument is exactly the same for the cover \(([t,t + 1], [t + 1, t+1])\).  
\end{proof}

%\fred{Let us focus on the category of graphs and explicit the pullback and pushouts of two graphs. Thus, let $G = (V, E(G)), H = (V,E(H)), K = (V,E(K))$ be three undirected graphs with the same sets of vertices. Suppose there are homomorphisms $h: G\to H$ and $k : G\to K$. Since $h$ and $k$ map sets of vertices to sets of vertices and $V = V(G) = V(H) = V(K)$, then by definition $h$ and $k$ are identities. Similarly, since $h$ and $k$ are homomorphisms, if $uv\in E(G)$, then $h(u)h(v) = uv\in E(H)$ and if $xy\in E(G)$, then $k(x)k(y) = xy\in E(K)$. The only graph $G$ that respects this property is the \emph{pullback} $G = (V, E(H)\cap E(K))$. By duality, the pushout of $H$ and $K$ is the graph $L = (V, E(H)\cup E(K))$ with homomorphisms $H\to L$ and $K\to L$. This observation will be important in Section~\ref{sec:examples-of-narratives} when we consider graph narratives.}

For an example of a cumulative narrative,  consider the following diagram (recall that, since they are co-sheaves, cumulative narratives are covariant functors). 
% https://q.uiver.app/#q=WzAsNixbMCwyLCJGXzFeMTo9XFx7YV8xLCBhXzIsYixjXFx9Il0sWzQsMiwiRl8zXjM6PVxce2FfXFxzdGFyLGInLCBjJ1xcfSJdLFsxLDEsIkZfMV4yID0gXFx7YV9cXHN0YXIsYiwgYicsY1xcfSJdLFsyLDIsIkZfMl4yOj1cXHthX1xcc3RhcixiJyxjXFx9Il0sWzMsMSwiRl8yXjMgPSBcXHthX1xcc3RhciwgYicsIGMsIGMnXFx9Il0sWzIsMCwiRl8xXjM9XFx7YV9cXHN0YXIsIGIsIGInLCBjLCBjJ1xcfSJdLFswLDIsIiIsMix7ImNvbG91ciI6WzI0MCw2MCw2MF19XSxbMywyLCIiLDIseyJzdHlsZSI6eyJ0YWlsIjp7Im5hbWUiOiJob29rIiwic2lkZSI6InRvcCJ9fX1dLFszLDQsIiIsMix7InN0eWxlIjp7InRhaWwiOnsibmFtZSI6Imhvb2siLCJzaWRlIjoidG9wIn19fV0sWzEsNCwiIiwyLHsic3R5bGUiOnsidGFpbCI6eyJuYW1lIjoiaG9vayIsInNpZGUiOiJ0b3AifX19XSxbMiw1LCIiLDIseyJzdHlsZSI6eyJ0YWlsIjp7Im5hbWUiOiJob29rIiwic2lkZSI6InRvcCJ9fX1dLFs0LDUsIiIsMix7InN0eWxlIjp7InRhaWwiOnsibmFtZSI6Imhvb2siLCJzaWRlIjoidG9wIn19fV0sWzUsMywiIiwxLHsic3R5bGUiOnsibmFtZSI6ImNvcm5lciJ9fV1d
\[
\adjustbox{scale=1.5, max width=\textwidth}{%,center}{
\begin{tikzcd}
	&& {F_1^3=\{a_\star, b, b', c, c'\}} \\
	& {F_1^2 = \{a_\star,b, b',c\}} && {F_2^3 = \{a_\star, b', c, c'\}} \\
	{F_1^1:=\{a_1, a_2,b,c\}} && {F_2^2:=\{a_\star,b',c\}} && {F_3^3:=\{a_\star,b', c'\}}
	\arrow[color={rgb,255:red,92;green,92;blue,214}, from=3-1, to=2-2]
	\arrow[hook, from=3-3, to=2-2]
	\arrow[hook, from=3-3, to=2-4]
	\arrow[hook, from=3-5, to=2-4]
	\arrow[hook, from=2-2, to=1-3]
	\arrow[hook, from=2-4, to=1-3]
	\arrow["\lrcorner"{anchor=center, pos=0.125, rotate=-45}, draw=none, from=1-3, to=3-3]
\end{tikzcd}
}
\]
We can think of this diagram (where we denoted injections via hooked arrows) as representing a cumulative view of the example from Section~\ref{sec:temporal-sets} of ice cream companies over time. Note that not all arrows are injections (the arrow \(F_1^1 \to F_1^2\) marked in blue is not injective since it takes every company to itself except for \(a_1\) and \(a_2\) which are both mapped to \(a_\star\)). Thus one can think of the cumulative perspective as accumulating not only the data (the companies) seen so far, but also the relationships that are `discovered' thus far in time. 

\subsection{Relating the Cumulative and Persistent Perspectives}\label{sec:cumulative-vs-persistent}

This section marks a significant step toward achieving our Desideratum~\ref{desideratum-2} for a theory for temporal structures. This desideratum emerges from the realization that, as we extend our focus to encompass categories beyond graphs, there exists a potential for information loss during the transition between the cumulative and persistent data structures. The present section systematically characterizes such transitions. Our Theorem \ref{thm:adjunction} yields two key results: the functoriality of transitioning from cumulative to persistent and vice versa, and the establishment of the adjunction \(\mathscr{K} \dashv \mathscr{P}\) formally linking these perspectives.

\begin{theorem}\label{thm:adjunction}
Let \(\cat{D}\) be a category with limits and colimits and $\cat{T}$ a time category. There exist functors \(\mathscr{K} \colon \mathsf{Pe}(\cat{T}, \mathsf{D}) \to \mathsf{Cu}(\cat{T}, \mathsf{D})\) and \(\mathscr{P} \colon \mathsf{Cu}(\cat{T}, \mathsf{D}) \to \mathsf{Pe}(\cat{T}, \mathsf{D})\).  Moreover, these functors are adjoint to each other:
% https://q.uiver.app/#q=WzAsMixbMCwwLCJcXG1hdGhzZntQZX0oXFxjYXR7VH0sIFxcbWF0aHNme0R9KSJdLFsyLDAsIlxcbWF0aHNme0N1fShcXGNhdHtUfSwgXFxtYXRoc2Z7RH0pIl0sWzAsMSwie1xcbWF0aHNjcntLfX0iLDAseyJjdXJ2ZSI6LTJ9XSxbMSwwLCJ7XFxtYXRoc2Nye1B9fSIsMCx7ImN1cnZlIjotMn1dLFsyLDMsIiIsMix7ImxldmVsIjoxLCJzdHlsZSI6eyJuYW1lIjoiYWRqdW5jdGlvbiIsImJvZHkiOnsibmFtZSI6Im5vbmUifSwiaGVhZCI6eyJuYW1lIjoibm9uZSJ9fX1dXQ==
\[\begin{tikzcd}
	{\mathsf{Pe}(\cat{T}, \mathsf{D})} && {\mathsf{Cu}(\cat{T}, \mathsf{D})}
	\arrow[""{name=0, anchor=center, inner sep=0}, "{{\mathscr{K}}}", curve={height=-12pt}, from=1-1, to=1-3]
	\arrow[""{name=1, anchor=center, inner sep=0}, "{{\mathscr{P}}}", curve={height=-12pt}, from=1-3, to=1-1]
	\arrow["\dashv"{anchor=center, rotate=-90}, draw=none, from=0, to=1]
\end{tikzcd}\]
\end{theorem}

\begin{proof} We first prove that passing from the persistent to the cumulative perspective is functorial.  To this end, we define a functor \(\mathscr{K} \colon \mathsf{Pe}(\cat{T}, \mathsf{D}) \to \mathsf{Cu}(\cat{T}, \mathsf{D})\) which takes any sheaf \(F \colon \cat{T}^{op} \to \cat{D}\) to the copresheaf \(\mathscr{K}(F) \colon \cat{T} \to \cat{D}\) defined on objects by:
\begin{align*}
     & [a,b] \mapsto \mathscr{K}(F)_a^b := \colim ((\cat{T}/[a,b] \hookrightarrow \cat{T})^{op} \xrightarrow{F} \cat{D}),
\end{align*} 
where \(\mathscr{K}(F)_a^b\) serves as shorthand for \(\mathscr{K}(F)([a,b])\).  The existence of the colimit \(\mathscr{K}(F)_a^b \) follows from the hypothesis, since \((\cat{T}/[a,b] \hookrightarrow \cat{T})^{op} \xrightarrow{F} \cat{D}\) is a diagram in \(\cat{D}\).  Moreover, \(\mathscr{K}(F)\) is defined on arrows as follows:
\begin{align*}
     &  \bigl ( [a',b'] \xhookrightarrow{f} [a,b] \bigr ) \quad \mapsto \quad \bigl ( \mathscr{K}(F)_{a'}^{b'} \xrightarrow{\mathscr{K}(F)f} \mathscr{K}(F)_a^b \bigr ),
\end{align*}
where \(\mathscr{K}(F)f\) is the unique arrow from \(\mathscr{K}(F)_{a'}^{b'}\) to \(\mathscr{K}(F)_a^b\), determined by the universal property of \(\mathscr{K}(F)_{a'}^{b'}\). The fact that \(\mathscr{K}(F)f\) maps identities to identities and respects composition also follows from universal properties of certain colimits involved.  Moreover, $\mathscr{K}(F)$ satisfies the sheaf condition by construction (since its values in any interval are determined by an appropriate colimit).\\

Now we prove that passing from the cumulative to the persistent perspective is functorial.  For this, we define \(\mathscr{P}\) as the map that assigns to any cosheaf  \(\hat{F} \colon \cat{T} \to \cat{D}\) in  \(\Cu\) the presheaf \(\mathscr{P}(\hat{F})\colon \cat{T}^{op} \to \cat{D}\) defined on objects by:
\begin{align*}
     & [a,b] \mapsto \mathscr{P}(\hat{F})_a^b := \lim (\cat{T}/[a,b] \hookrightarrow \cat{T} \xrightarrow{\hat{F}} \cat{D}).
\end{align*}
We will use the notation \(\mathscr{P}(\hat{F})_a^b\) instead of \(\mathscr{P}(\hat{F})([a,b])\).  The existence of \(\lim (\cat{T}/[a,b] \hookrightarrow \cat{T} \xrightarrow{\hat{F}} \cat{D})\) follows from the hypothesis, since  \(\cat{T}/[a,b] \hookrightarrow \cat{T} \xrightarrow{\hat{F}} \cat{D}\) is a diagram in \(\cat{D}\).   Furthermore, \(\mathscr{P}(\hat{F})\) is defined on the arrows as follows:
\[
     \bigl ( [a',b'] \xhookrightarrow{f} [a,b] \bigr) \quad \mapsto \quad \bigl ( \mathscr{P}(\hat{F})_a^b \xrightarrow{\mathscr{P}(\hat{F})f} \mathscr{P}(\hat{F})_{a'}^{b'} \bigr),
\]
where \(\mathscr{P}(\hat{F})f\) is the unique arrow from \(\mathscr{P}(\hat{F})_a^b\) to \(\mathscr{P}(\hat{F})_{a'}^{b'}\), determined by the universal property of \(\mathscr{P}(\hat{F})_{a'}^{b'}\). The fact that \(\mathscr{P}(\hat{F})\) maps identities to identities and respects composition follows from universal properties of certain limits involved.  Moreover, $\mathscr{P}(\hat{F})$ satisfies the sheaf condition by construction (since its values in any interval are determined by an appropriate limit).\\

We will now prove that there exist an adjunction \(\mathscr{K} \dashv \mathscr{P}\), which relates the two perspectives. For this, we will build a pair of natural transformations \(\mathscr{KP}\xrightarrow{\mathsf{\epsilon}} \cat{1}_{\mathsf{Cu(\cat{T},D)}} \) and \(\cat{1}_{\mathsf{Pe(\cat{T},D)}} \xrightarrow{\mathsf{\eta}} \mathscr{PK}\) that make the triangle identities commute:

% https://q.uiver.app/#q=WzAsNixbMCwwLCJcXG1hdGhjYWx7UH0iXSxbMiwwLCJcXG1hdGhjYWx7UEtQfSJdLFsyLDIsIlxcbWF0aGNhbHtQfSJdLFs0LDAsIlxcbWF0aGNhbHtLfSJdLFs2LDAsIlxcbWF0aGNhbHtLUEt9Il0sWzYsMiwiXFxtYXRoY2Fse0t9Il0sWzAsMSwiXFxldGFcXG1hdGhjYWx7UH0iXSxbMSwyLCJcXG1hdGhjYWx7UH1cXGVwc2lsb24iXSxbMCwyLCJcXG1hdGhzZnsxfV97XFxtYXRoY2Fse1B9fSIsMl0sWzMsNCwiXFxtYXRoY2Fse0t9XFxldGEiXSxbNCw1LCJcXGVwc2lsb25cXG1hdGhjYWx7S30iXSxbMyw1LCJcXG1hdGhzZnsxfV97XFxtYXRoY2Fse0t9fSIsMl1d
\[\begin{tikzcd}
	{\mathscr{P}} && {\mathscr{PKP}} && {\mathscr{K}} && {\mathscr{KPK}} \\
	\\
	&& {\mathscr{P}} &&&& {\mathscr{K}}
	\arrow["{\eta\mathscr{P}}", from=1-1, to=1-3]
	\arrow["{\mathsf{1}_{\mathscr{P}}}"', from=1-1, to=3-3]
	\arrow["{\mathscr{P}\epsilon}", from=1-3, to=3-3]
	\arrow["{\mathscr{K}\eta}", from=1-5, to=1-7]
	\arrow["{\mathsf{1}_{\mathscr{K}}}"', from=1-5, to=3-7]
	\arrow["{\epsilon\mathscr{K}}", from=1-7, to=3-7]
\end{tikzcd}\]

% \hat{F} \colon I \to \cat{D} 
% \mathsf{Cu(I,D)}
% \mathscr{PKP}
We need to define the components \(\mathscr{KP}(\hat{F}) \xrightarrow{\mathsf{\epsilon}_{\hat{F}}} \cat{1}_{\mathsf{Cu(\cat{T},D)}}(\hat{F}) \) for every cosheaf in \(\mathsf{Cu(\cat{T},D)}\). This involves choosing natural transformations \(\epsilon_{\hat{F}_a^b} : \mathscr{KP}(\hat{F})_a^b \to \hat{F}_a^b\) for each interval \([a,b]\) in \(\cat{T}\). As \(\mathscr{KP}(\hat{F})_a^b\) is a colimit, there exists only one such arrow. We define \(\epsilon_{\hat{F}_a^b}\) to be this unique arrow, as illustrated in the cummutative diagram on the left:

% https://q.uiver.app/#q=WzAsMTAsWzEsNCwiXFxtYXRoY2Fse1B9KFxcaGF0e0Z9KV9hXmIiXSxbMiwzLCIgXFxoYXR7Rn1fYl5iID0gXFxtYXRoY2Fse1B9KFxcaGF0e0Z9KV9iXmIiXSxbMCwzLCJcXGhhdHtGfV9hXmEgPSBcXG1hdGhjYWx7UH0oXFxoYXR7Rn0pX2FeYSJdLFsxLDIsIlxcbWF0aGNhbHtLfVAoXFxoYXR7Rn0pX2FeYiJdLFsxLDAsIiBcXGhhdHtGfV9hXmIiXSxbNCwzLCJGX2FeYSA9IFxcbWF0aGNhbHtLfShGKV9hXmEiXSxbNSwyLCJcXG1hdGhjYWx7UH1cXG1hdGhjYWx7S30oRilfYV5iIl0sWzUsNCwiXFxtYXRoY2Fse0t9KEYpX2FeYiJdLFs2LDMsIkZfYl5iID0gXFxtYXRoY2Fse0t9KEYpX2JeYiJdLFs1LDAsIkZfYV5iIl0sWzAsMV0sWzAsMl0sWzIsM10sWzEsM10sWzMsMCwiIiwxLHsic3R5bGUiOnsibmFtZSI6ImNvcm5lciJ9fV0sWzEsNCwiIiwxLHsiY3VydmUiOjJ9XSxbMiw0LCIiLDEseyJjdXJ2ZSI6LTJ9XSxbMyw0LCJcXGVwc2lsb25fe1xcaGF0e0Z9X2FeYn0iLDEseyJjb2xvdXIiOlswLDYwLDYwXSwic3R5bGUiOnsiYm9keSI6eyJuYW1lIjoiZGFzaGVkIn19fSxbMCw2MCw2MCwxXV0sWzYsNV0sWzUsN10sWzgsN10sWzYsOF0sWzksNSwiIiwxLHsiY3VydmUiOjJ9XSxbOSw4LCIiLDEseyJjdXJ2ZSI6LTJ9XSxbOSw2LCJcXGV0YV97Rl9hXmJ9IiwxLHsiY29sb3VyIjpbMCw2MCw2MF0sInN0eWxlIjp7ImJvZHkiOnsibmFtZSI6ImRhc2hlZCJ9fX0sWzAsNjAsNjAsMV1dLFs2LDcsIiIsMSx7InN0eWxlIjp7Im5hbWUiOiJjb3JuZXIifX1dXQ==
\[\begin{tikzcd}[column sep=small]
	& { \hat{F}_a^b} &&&& {F_a^b} \\
	\\
	& {\mathscr{KP}(\hat{F})_a^b} &&&& {\mathscr{P}\mathscr{K}(F)_a^b} \\
	{\hat{F}_a^a = \mathscr{P}(\hat{F})_a^a} && { \hat{F}_b^b = \mathscr{P}(\hat{F})_b^b} && {F_a^a = \mathscr{K}(F)_a^a} && {F_b^b = \mathscr{K}(F)_b^b} \\
	& {\mathscr{P}(\hat{F})_a^b} &&&& {\mathscr{K}(F)_a^b}
	\arrow["{\eta_{F_a^b}}"{description}, color={rgb,255:red,214;green,92;blue,92}, dashed, from=1-6, to=3-6]
	\arrow[curve={height=12pt}, from=1-6, to=4-5]
	\arrow[curve={height=-12pt}, from=1-6, to=4-7]
	\arrow["{\epsilon_{\hat{F}_a^b}}"{description}, color={rgb,255:red,214;green,92;blue,92}, dashed, from=3-2, to=1-2]
	\arrow["\lrcorner"{anchor=center, pos=0.125, rotate=-45}, draw=none, from=3-2, to=5-2]
	\arrow[from=3-6, to=4-5]
	\arrow[from=3-6, to=4-7]
	\arrow["\lrcorner"{anchor=center, pos=0.125, rotate=-45}, draw=none, from=3-6, to=5-6]
	\arrow[curve={height=-12pt}, from=4-1, to=1-2]
	\arrow[from=4-1, to=3-2]
	\arrow[curve={height=12pt}, from=4-3, to=1-2]
	\arrow[from=4-3, to=3-2]
	\arrow[from=4-5, to=5-6]
	\arrow[from=4-7, to=5-6]
	\arrow[from=5-2, to=4-1]
	\arrow[from=5-2, to=4-3]
\end{tikzcd}\]

Applying a dual argument, we can construct \(\cat{1}_{\mathsf{Pe(\cat{T},D)}} \xrightarrow{\mathsf{\eta}} \mathscr{PK}\) using the natural transformations \(\eta_{F_a^b}\), as illustrated in the diagram on the right. The existence of these natural transformations \(\epsilon\) and \(\eta\) is sufficient to ensure that the triangle identities commute. This is attributed to the universal map properties of \( \mathscr{KP}(\hat{F})_a^b \) and \( \mathscr{P}\mathscr{K}(F)_a^b \), respectively.
\end{proof}

From a practical perspective, Theorem~\ref{thm:adjunction} implies that in general there is the potential for a loss of information when one passes from one perspective (the persistent one, say) to another (the cumulative one) and back again. Furthermore the precise way in which this information may be lost is explicitly codified by the unit \(\eta\) and co-unit \(\epsilon\) of the adjunction. These observations, which were hidden in other encodings of temporal data~\cite{OthonTempSurvey, KEMPE2002820, FlocchiniQuattrocchiSantoroCasteigts}, are of great practical relevance since it means that one must take a great deal of care when collecting temporal data: the choices of mathematical representations may not be interchangeable.

\subsection{Collecting Examples: Narratives are Everywhere}\label{sec:examples-of-narratives}
\paragraph{Temporal graphs.} Think of satellites orbiting around the earth where, at each given time, the distance between any two given satellites determines their ability to communicate. To understand whether a signal can be sent from one satellite to another one needs a temporal graph: it does not suffice to solely know the static structure of the time-indexed communication networks between these satellites, but instead one needs to also keep track of the relationships between these snapshots. We can achieve this with narratives of graphs, namely cosheaves (or sheaves, if one is interested in the persistent model) of the form $\Ga \colon \cat{T} \to \cat{Grph}$ from a time category $\cat{T}$ into \(\cat{Grph}\), \textit{a} category of graphs.\footnote{Note that many categories of graphs are presheaf toposes and thus have all limits and colimits.} There are many ways in which one could define categories of graphs; for the purposes of recovering definitions from the literature we will now briefly review the category of graphs we choose to work with. 

We view graphs as objects in \(\cat{Set}^{\cat{SGr}}\), the functor category from the graph \textit{schema} to set. It has as objects functors \(G \colon \cat{SGr} \to \cat{Set}\) where \(\cat{SGr}\) is thought of as a \textit{schema} category with only two objects called \(E\) and \(V\) and two non-identity morphisms \(s, t \colon E \to V\) which should be thought as mnemonics for `source' and `target'. We claim that \(\cat{Set}^{\cat{SGr}}\) is the category of directed multigraphs and graph homomorphisms. To see this, notice that any functor \(G \colon \cat{SGr} \to \cat{Set}\) consists of two sets: \(G(E)\) (the edge set) and \(G(V)\) (the vertex set). Moreover each edge \(e \in G(E)\) gets mapped to two vertices (namely its source \(G(s)(e)\) and target \(G(t)(e)\)) via the functions \(G(s) \colon G(E) \to G(V)\) and \(G(t) \colon G(E) \to G(V)\). Arrows in $\cat{Set}^{\cat{SGr}}$ are natural transformations between functors.  To see that natural transformations \(\eta \colon G \Rightarrow H\) define graph homomorphisms, note that any such \(\eta\) consists of functions  \(\eta_E \colon G(E) \to H(E)\) and \(\eta_V \colon G(V) \to H(V)\) (its components at \(E\) and \(V\)) which commute with the source and target maps of \(G\) and \(H\). 

The simplest definition of temporal graphs in the literature is that due to Kempe, Kleinberg and Kumar~\cite{KEMPE2002820} which views temporal graphs as a sequence of edge sets over a fixed vertex set. 

\newsavebox\Gzero
\sbox\Gzero{
\begin{tikzpicture}[auto, circle, >=stealth, every node/.style={draw=black, inner sep=1pt}, scale=0.5, fill=black]
\draw (0,0) node[label=135:$a$] (va) {};
\draw (0,1) node[label=135:$b$] (vb) {};
\draw (1,1) node[label=45:$c$] (vc) {};
\draw (1,0) node[label=45:$d$] (vd) {};

\path[draw, thick] (va) -- (vb) -- (vc) -- (vd) -- (va);
\end{tikzpicture}
}
\newsavebox\Gone
\sbox\Gone{
\begin{tikzpicture}[auto, circle, >=stealth, every node/.style={draw=black, inner sep=1pt}, scale=0.5, fill=black]
\draw (0,0) node[label=135:$a$] (va) {};
\draw (0,1) node[label=135:$b$] (vb) {};
\draw (1,1) node[label=45:$c$] (vc) {};
\draw (1,0) node[label=45:$d$] (vd) {};

\path[draw, thick] (va) -- (vc);
\path[draw, thick] (vb) -- (vd);
\end{tikzpicture}
}
\newsavebox\Gtwo
\sbox\Gtwo{
\begin{tikzpicture}[auto, circle, >=stealth, every node/.style={draw=black, inner sep=1pt}, scale=0.5, fill=black]
\draw (0,0) node[label=135:$a$] (va) {};
\draw (0,1) node[label=135:$b$] (vb) {};
\draw (1,1) node[label=45:$c$] (vc) {};
\draw (1,0) node[label=45:$d$] (vd) {};

\path[draw, thick] (va) -- (vc);
\path[draw, thick] (va) -- (vd);
\path[draw, thick] (vb) -- (vc);
\path[draw, thick] (vb) -- (vd);
\end{tikzpicture}
}
\newsavebox\Gempty
\sbox\Gempty{
\begin{tikzpicture}[auto, circle, >=stealth, every node/.style={draw=black, inner sep=1pt}, scale=0.5, fill=black]
\draw (0,0) node[label=135:$a$] (va) {};
\draw (0,1) node[label=135:$b$] (vb) {};
\draw (1,1) node[label=45:$c$] (vc) {};
\draw (1,0) node[label=45:$d$] (vd) {};
\end{tikzpicture}
}

\newsavebox\Gfull
\sbox\Gfull{
\begin{tikzpicture}[auto, circle, >=stealth, every node/.style={draw=black, inner sep=1pt}, scale=0.5, fill=black]
\draw (0,0) node[label=135:$a$] (va) {};
\draw (0,1) node[label=135:$b$] (vb) {};
\draw (1,1) node[label=45:$c$] (vc) {};
\draw (1,0) node[label=45:$d$] (vd) {};

\path[draw, thick] (va) -- (vb);
\path[draw, thick] (va) -- (vc);
\path[draw, thick] (va) -- (vd);
\path[draw, thick] (vb) -- (vc);
\path[draw, thick] (vb) -- (vd);
\path[draw, thick] (vc) -- (vd);
\end{tikzpicture}
}

\newsavebox\MGfull
\sbox\MGfull{
\begin{tikzpicture}[auto, circle, >=stealth, every node/.style={draw=black, inner sep=1pt}, scale=0.5, fill=black]
\draw (0,0) node[label=135:$a$] (va) {};
\draw (0,1) node[label=135:$b$] (vb) {};
\draw (1,1) node[label=45:$c$] (vc) {};
\draw (1,0) node[label=45:$d$] (vd) {};

\path[draw, thick] (va) -- (vb);
\path[draw, thick] (va) -- (vc);
\draw[thick] (va) edge[bend right] (vd);
\draw[thick] (vd) edge[bend right] (va);
\draw[thick] (vb) edge[bend right] (vc);
\draw[thick] (vc) edge[bend right] (vb);
\path[draw, thick] (vb) -- (vd);
\path[draw, thick] (vc) -- (vd);
\end{tikzpicture}
}
\begin{figure}
\begin{subfigure}[ht]{0.75\textwidth}
\centering
\begin{tikzpicture}[auto, circle, >=stealth, every node/.style={draw=black, inner sep=1pt}, scale=1, fill=black]
\draw (0,0) node[label=135:$a$] (va) {};
\draw (0,1) node[label=135:$b$] (vb) {};
\draw (1,1) node[label=45:$c$] (vc) {};
\draw (1,0) node[label=45:$d$] (vd) {};

\path[draw, thick] (va) -- (vb) -- (vc) -- (vd) -- (va);

\node[draw = none] at (0.5, -0.25) {$G_0$};
\end{tikzpicture}
\begin{tikzpicture}[auto, circle, >=stealth, every node/.style={draw=black, inner sep=1pt}, scale=1, fill=black]
\draw (0,0) node[label=135:$a$] (va) {};
\draw (0,1) node[label=135:$b$] (vb) {};
\draw (1,1) node[label=45:$c$] (vc) {};
\draw (1,0) node[label=45:$d$] (vd) {};

\path[draw, thick] (va) -- (vc);
\path[draw, thick] (vb) -- (vd);

\node[draw = none] at (0.5, -0.25) {$G_1$};
\end{tikzpicture}
\begin{tikzpicture}[auto, circle, >=stealth, every node/.style={draw=black, inner sep=1pt}, scale=1, fill=black]
\draw (0,0) node[label=135:$a$] (va) {};
\draw (0,1) node[label=135:$b$] (vb) {};
\draw (1,1) node[label=45:$c$] (vc) {};
\draw (1,0) node[label=45:$d$] (vd) {};

\path[draw, thick] (va) -- (vc);
\path[draw, thick] (va) -- (vd);
\path[draw, thick] (vb) -- (vc);
\path[draw, thick] (vb) -- (vd);

\node[draw = none] at (0.5, -0.25) {$G_2$};
\end{tikzpicture}
\caption{A temporal graph $\Ga$ (in the sense of Definition~\ref{def:simple_definition_of_temporal_graph}) with  three snapshots}
\label{fig:temp_graph_example}
\end{subfigure}
\centering
\begin{subfigure}[ht]{0.4\textwidth}
\centering
\begin{tikzpicture}[auto, circle, >=stealth, every node/.style={draw=black, inner sep=1pt}, scale=1, fill=black]
\draw (0,0) node (v00) {\scalebox{0.5}{\usebox{\Gempty}}};
\draw (-1,-1) node (vm1m1) {\scalebox{0.5}{\usebox{\Gempty}}};
\draw (1,-1) node (v1m1) {\scalebox{0.5}{\usebox{\Gone}}};
\draw (-2,-2) node (vm2m2) {\scalebox{0.5}{\usebox{\Gzero}}};
\draw (0,-2) node (v0m2) {\scalebox{0.5}{\usebox{\Gone}}};
\draw (2,-2) node (v2m2) {\scalebox{0.5}{\usebox{\Gtwo}}};

\draw[->] (v00) -- (vm1m1);
\draw[->] (v00) -- (v1m1);
\draw[->] (vm1m1) -- (vm2m2);
\draw[->] (vm1m1) -- (v0m2);
\draw[->] (v1m1) -- (v0m2);
\draw[->] (v1m1) -- (v2m2);
\end{tikzpicture}
\caption{The persistent narrative of $\Ga$}
\label{fig:pernar_example}
\end{subfigure}
\hfill
\begin{subfigure}[ht]{0.4\textwidth}
\centering
\begin{tikzpicture}[auto, circle, >=stealth, every node/.style={draw=black, inner sep=1pt}, scale=1, fill=black]
\draw (0,0) node (v00) {\scalebox{0.5}{\usebox{\MGfull}}};
\draw (-1,-1) node (vm1m1) {\scalebox{0.5}{\usebox{\Gfull}}};
\draw (1,-1) node (v1m1) {\scalebox{0.5}{\usebox{\Gtwo}}};
\draw (-2,-2) node (vm2m2) {\scalebox{0.5}{\usebox{\Gzero}}};
\draw (0,-2) node (v0m2) {\scalebox{0.5}{\usebox{\Gone}}};
\draw (2,-2) node (v2m2) {\scalebox{0.5}{\usebox{\Gtwo}}};

\draw[<-] (v00) -- (vm1m1);
\draw[<-] (v00) -- (v1m1);
\draw[<-] (vm1m1) -- (vm2m2);
\draw[<-] (vm1m1) -- (v0m2);
\draw[<-] (v1m1) -- (v0m2);
\draw[<-] (v1m1) -- (v2m2);
\end{tikzpicture}
\caption{The cumulative narrative of $\Ga$}
\label{fig:cunar_example}
\end{subfigure}
\caption{A temporal graph along with its persistent and cumulative narratives}
\label{fig:encoding-temp-grphs-example}
\end{figure}

\begin{definition}[\cite{KEMPE2002820}]\label{def:simple_definition_of_temporal_graph}
A temporal graph $\mathcal{G}$ consists of a pair $\bigl (V,(E_i)_{i \in \mathbb{N}} \bigr)$ where \(V\) is a set and \((E_i)_{i \in \mathbb{N}}\) is a sequence of binary relations on \(V\).
\end{definition}

The above definition can be immediately formulated in terms of our discrete cumulative (resp. persistent) graph narratives whereby a temporal graph is a cumulative narrative valued in the category \(\cat{Set}^{\cat{SGr}}\) with discrete time. To see this, observe that, since Definition~\ref{def:simple_definition_of_temporal_graph} assumes a fixed vertex set and since it assumes simple graphs, the cospans (resp. spans) can be inferred from the snapshots (see Figure~\ref{fig:encoding-temp-grphs-example} for examples). For instance, in the persistent case, there is one maximum common subgraph to use as the apex of each span associated to the inclusions of intervals of length zero into intervals of length one. This, combined with Proposition~\ref{prop:specifying-a-discrete-narrative} yields a unique persistent graph narrative which encodes any given temporal graph (as given in Definition~\ref{def:simple_definition_of_temporal_graph}). 

Notice that once an edge or vertex disappears in a persistent (or cumulative) graph narrative, it can never reappear: the only way to reconnect two vertices is to create an entirely new edge. In particular this means that cumulative graph narratives associate to most intervals of time a multigraph rather than a simple graph (see Figure~\ref{fig:cunar_example}). This is a very natural requirement, for instance: imagining a good being delivered from $u$ to $v$ at times $t$ and $t'$, it is clear that the goods need not be delivered by the same person and, in any event, the very acts of delivery are different occurrences.\footnote{If one insists on avoiding this ``duplication'' of edge- or vertex-appearances, then this can be achieved by changing the codomain of these narratives to another category of graphs (for instance to a category of simple, reflexive graphs or to a slice category of some large ambient graph).}

As shown by Patterson, Lynch and Fairbanks~\cite{Patterson2022categoricaldata}, by passing to slice categories, one can furthermore encode various categories of labelled data. For instance, one can fix the monoid of natural numbers viewed as a single-vertex graph with a loop edge for each natural number \(G_{B\mathbb{N}}~\colon~\cat{SGr}~\to~\cat{Set}\) having \(G_{B\mathbb{N}}(V) = 1\) and \(G_{B\mathbb{N}}(E) = \mathbb{N})\) and consider the slice category \(\cat{Set}^{\cat{SGr}} / G_{B\mathbb{N}}\). This will have pairs \((G, \lambda \colon G \to G_{B\mathbb{N}})\) as objects where \(G\) is a graph and \(\lambda\) is a graph homomorphism effectively assigning a natural number label to each edge of \(G\). The morphisms of \(\cat{Set}^{\cat{SGr}} / G_{B\mathbb{N}}\) are label-preserving graph homomorphisms. Thus narratives valued in \(\cat{Set}^{\cat{SGr}} / G_{B\mathbb{N}}\) can be interpreted as time-varying graphs whose edges come equipped with \textbf{latencies} (which can change with time). 

By similar arguments, it can be easily shown that one can encode categories of graphs which have labeled vertices and labeled edges~\cite{Patterson2022categoricaldata}. Narratives in such categories correspond to time-varying graphs equipped with both \textbf{vertex- and edge-latencies}. This allows us to recover the following notion, due to Casteigts, Flocchini, Quattrociocchi and Santoro, of a time-varying graph which has recently attracted much attention in the literature.  

\begin{definition}[Section 2 in \cite{FlocchiniQuattrocchiSantoroCasteigts}]\label{def:general_definition_of_temporal_graph}
Take $\mathbb{T}$ to be either $\mathbb{N}$ or $\mathbb{R}$. A \define{$\mathbb{T}$-temporal (directed) network} is a quintuple $(G, \rho_e, \eta_e, \rho_v, \eta_v)$ where $G$ is a (directed) graph and $\rho_e$, $\eta_e$, $\rho_v$ and $\eta_v$ are functions of the following types: 
\begin{align*}
    \rho_e &: E(G) \times \mathbb{T} \to \{\bot, \top\}, &\eta_e : E(G) \times \mathbb{T} \to \mathbb{T}, \\ 
    \rho_v &: V(G) \times \mathbb{T} \to \{\bot, \top\}, &\eta_v : V(G) \times \mathbb{T} \to \mathbb{T} 
\end{align*}
where \(\rho_e\) and \(\rho_v\) are are functions indicating whether an edge or vertex is active at a given time and where \(\eta_e\) and \(\eta_v\) are latency functions indicating the amount of time required to traverse an edge or vertex. 
\end{definition}

We point out that this definition, stated as in~\cite{FlocchiniQuattrocchiSantoroCasteigts} does \textit{not enforce any coherence conditions} to ensure that edges are present at times in which their endpoints are. Our approach, in contrast, comes immediately equipped with all such necessary coherence conditions.

\paragraph{Other structures.}

There exist diverse types of graphs, such as reflexive, symmetric, and half-edge graphs, each characterized by the nature of the relation aimed to be modeled. Each graph type assemble into specific categories, and the selection of graph categories distinctly shapes the resulting graph narratives. To systematically investigate the construction of various graph narratives, we employ a category-theoretic trick. This involves encoding these diverse graphs as functors, specifically set-valued copresheaves, over a domain category known as a schema. The schema encapsulates the syntax of a particular graph type (e.g., symmetric graphs, reflexive graphs, etc.), allowing us to encode a multitude of structures. Notable examples of such schemata include \(\cat{SSGr}\), reflexive graphs \(\cat{SRGr}\), symmetric-and-reflexive graphs \(\cat{SSRGr}\) and half-edge graphs \(\cat{SHeGr}\).

% https://tikzcd.yichuanshen.de/#N4Igdg9gJgpgziAXAbVABwnAlgFyxMJZABgBpiBdUkANwEMAbAVxiRAFEQBfU9TXfIRRkAjFVqMWbAGrdeIDNjwEiANnLj6zVohAAJOXyWC1pMdS1TdsnkYErhpAEybJOkAB0POGAA8cwHAAdDhBAARcYXBhXgDGWABOsWFYYQC8YTgx3n4BYXRgUBGZcYnJqRkItgr8ykLIThoWbmyc1Yr29Y3mEtoyhjXGDg3Orn26Xj7+gSHhkdGlSdkwDAzpmdnxS14rDAMddUQALE29Vhz7tSYoJz2W7jbyB9fIJy7N454508GhxQseLbldZZSa5YD5QrFHCLYGVbJTPIFIrzTZlZarEFo7YeXaXIb1dTvM7uME-WbFLBgGjY8rU9ZeAC2dBwAAs4AAzYBYKBcAD6Bi44hgUAA5vAiKAOQkIIykABmag4CBIRoktisarS2WqpUqxBqhgQCBoIgiAAcZA5jDgMHEDDoACMVgAFK4OEAJLCi1k4EAfc5UmgDbVyxBkEDKpARo0ms2W0jWhi2+1O13uoSe72+-3q3RYEMysMRqOIETUZ2FJAAWnlEfubAQ1AdzoYboJbC9Pr9WqL0b1SHLIErUAV9Zauh78lDSAArAPEOphzAq4hq0cAJwB9xNkAt9Md3RdnO9nWIeeR-VLkdITfbthTqV9xcLi8Nia41aFs-mhe-vfGqaKAWlaNp2s2aZthmnbZn6975t+Yb-qWW7Lqu66oe+IC7vuUGHlm3aIUgyH6qhN6IHeeaRkRiCoaW-5YTsX4Qa27adDBhFClwQA
\[
\adjustbox{scale=1.5, max width=\textwidth}{%,center}{
\begin{tikzcd}
E \arrow["i"', loop, distance=2em, in=125, out=55] \arrow[d, "s"', bend right] \arrow[d, "t", bend left] & E \arrow[d, "s"', bend right=49] \arrow[d, "t", bend left=49] & E \arrow["i"', loop, distance=2em, in=125, out=55] \arrow[d, "s"', bend right=49] \arrow[d, "t", bend left=49] & H \arrow["inv"', loop, distance=2em, in=125, out=55] \\
V                                                                                                        & V \arrow[u, "\ell"]                                           & V \arrow[u, "\ell"']                                                                                           & V \arrow[u, "e"]                                     \\
\text{s.t. } s \circ i = t \text{ and } t\circ i = s                                                     & \text{s.t. } s \circ \ell = t \circ \ell                      & \text{s.t. } s \circ i = t \text{ and } t\circ i = s \text{ and } s \circ \ell = t \circ \ell                  & \text{s.t. } \mathsf{inv} \circ \mathsf{inv} = \mathsf{id}_H          
\end{tikzcd}
}
\]

These are all subcategories of multigraphs but other relational structures of higher order such as Petri nets and simplicial complexes can also be constructed using this approach.  For instance, the following is the schema for Petri nets~\cite{Patterson2022categoricaldata}:

% https://q.uiver.app/#q=WzAsNSxbMSwwLCJcXG1hdGhzZntJbnB1dH0iXSxbMSwxLCJcXG1hdGhzZntTcGVjaWVzfSJdLFsxLDIsIlxcbWF0aHNme091dHB1dH0iXSxbMiwxLCJcXG1hdGhzZntUcmFuc2l0aW9ufSJdLFswLDEsIlxcbWF0aHNme1Rva2VufSJdLFswLDFdLFsyLDFdLFs0LDFdLFswLDNdLFsyLDNdXQ==&macro_url=https%3A%2F%2Fq.uiver.app%2F%23q%3DWzAsNSxbMCwxLCJGXzFeMTo9XFx7YV8xLCBhXzIsIGIsIGNcXH0iXSxbNCwxLCJGXzNeMzo9XFx7YV9cXHN0YXIsYicsIGMnXFx9Il0sWzEsMCwiRl8xXjIgPSBcXHthXzEsIGFfMixjXFx9Il0sWzIsMSwiRl8yXjI6PVxce2FfXFxzdGFyLGInLGNcXH0iXSxbMywwLCJGXzJeMyA9IFxce2FfXFxzdGFyLCBiJ1xcfSJdLFsyLDAsImZfezEsMn1eMSJdLFsyLDMsImZfezEsMn1eMiIsMl0sWzQsMywiZl97MiwzfV4yIl0sWzQsMSwiZl97MiwzfV4zIiwyXV0%3D
\[\begin{tikzcd}
	& {\mathsf{Input}} \\
	{\mathsf{Token}} & {\mathsf{Species}} & {\mathsf{Transition}} \\
	& {\mathsf{Output}}
	\arrow[from=1-2, to=2-2]
	\arrow[from=1-2, to=2-3]
	\arrow[from=2-1, to=2-2]
	\arrow[from=3-2, to=2-2]
	\arrow[from=3-2, to=2-3]
\end{tikzcd}\]

It is known that all of these categories of \(\cat{CSets}\) are \textit{presheaf toposes} (and thus admit limits and colimits which are computed point-wise) and thus we can define narratives as presheaves \(F \colon \cat{T}^{op} \to \cat{CSet}\) satisfying the sheaf condition stated in Proposition~\ref{prop:def:sheaves} for any choice of schema (e.g., \(\cat{SSGr}\), \(\cat{SRGr}\), \(\cat{SSRGr}\) \(\cat{SHeGr}\), etc.).

\begin{note}[Beyond relational structures]
Proposition~\ref{prop:def:sheaves} indeed states that  we can define narratives valued in any category that has limits and/or colimits. For instance, the category \(\cat{Met}\) of metric spaces and contractions is a complete category, allowing us to study persistent \(\cat{Met}\)-narratives. Diagram~\ref{diagram:companies-metric-spaces} illustrates a \(\cat{Met}\)-narrative that recounts the story of how the geographical distances of ice cream companies in Venice changed over time. Each snapshot (depicted in pink) represents a metric space, and all morphisms are canonical isometries. The curious reader can use it to speculate about why company $b$ ceased its activities and what happened to the physical facilities of companies $a_1$ and $c$.
\begin{equation}\label{diagram:companies-metric-spaces}
    \adjustbox{scale=1.5, max width=.95\textwidth}{%,center}{

\tikzset{every picture/.style={line width=0.75pt}} %set default line width to 0.75pt        

\begin{tikzpicture}[x=0.75pt,y=0.75pt,yscale=-1,xscale=1]
%uncomment if require: \path (0,165); %set diagram left start at 0, and has height of 165

%Rounded Rect [id:dp9494842872735878] 
\draw  [color={rgb, 255:red, 255; green, 255; blue, 255 }  ,draw opacity=1 ][fill={rgb, 255:red, 255; green, 192; blue, 203 }  ,fill opacity=1 ] (19.54,110.16) .. controls (19.54,103.96) and (24.56,98.93) .. (30.77,98.93) -- (132.44,98.93) .. controls (138.65,98.93) and (143.67,103.96) .. (143.67,110.16) -- (143.67,143.85) .. controls (143.67,150.05) and (138.65,155.08) .. (132.44,155.08) -- (30.77,155.08) .. controls (24.56,155.08) and (19.54,150.05) .. (19.54,143.85) -- cycle ;

%Straight Lines [id:da7056933848921817] 
\draw    (457.62,62.4) -- (443.07,78.46) -- (430.69,91.42) ;
\draw [shift={(429.31,92.86)}, rotate = 313.69] [color={rgb, 255:red, 0; green, 0; blue, 0 }  ][line width=0.75]    (10.93,-4.9) .. controls (6.95,-2.3) and (3.31,-0.67) .. (0,0) .. controls (3.31,0.67) and (6.95,2.3) .. (10.93,4.9)   ;
%Straight Lines [id:da04794267732743929] 
\draw    (179.61,62.14) -- (153.41,90.72) ;
\draw [shift={(152.06,92.2)}, rotate = 312.5] [color={rgb, 255:red, 0; green, 0; blue, 0 }  ][line width=0.75]    (10.93,-4.9) .. controls (6.95,-2.3) and (3.31,-0.67) .. (0,0) .. controls (3.31,0.67) and (6.95,2.3) .. (10.93,4.9)   ;
%Straight Lines [id:da9256611021932947] 
\draw    (260.32,61.88) -- (287.24,91.74) ;
\draw [shift={(288.58,93.22)}, rotate = 227.96] [color={rgb, 255:red, 0; green, 0; blue, 0 }  ][line width=0.75]    (10.93,-4.9) .. controls (6.95,-2.3) and (3.31,-0.67) .. (0,0) .. controls (3.31,0.67) and (6.95,2.3) .. (10.93,4.9)   ;
%Straight Lines [id:da33396431136558724] 
\draw    (534.5,62.65) -- (561.64,91.76) ;
\draw [shift={(563,93.22)}, rotate = 227] [color={rgb, 255:red, 0; green, 0; blue, 0 }  ][line width=0.75]    (10.93,-4.9) .. controls (6.95,-2.3) and (3.31,-0.67) .. (0,0) .. controls (3.31,0.67) and (6.95,2.3) .. (10.93,4.9)   ;
%Rounded Rect [id:dp25967434304965376] 
\draw  [color={rgb, 255:red, 255; green, 255; blue, 255 }  ,draw opacity=1 ][fill={rgb, 255:red, 255; green, 192; blue, 203 }  ,fill opacity=1 ] (157.49,12.82) .. controls (157.49,6.62) and (162.52,1.59) .. (168.72,1.59) -- (270.4,1.59) .. controls (276.6,1.59) and (281.63,6.62) .. (281.63,12.82) -- (281.63,46.51) .. controls (281.63,52.72) and (276.6,57.74) .. (270.4,57.74) -- (168.72,57.74) .. controls (162.52,57.74) and (157.49,52.72) .. (157.49,46.51) -- cycle ;

%Rounded Rect [id:dp7934017799582207] 
\draw  [color={rgb, 255:red, 255; green, 255; blue, 255 }  ,draw opacity=1 ][fill={rgb, 255:red, 255; green, 192; blue, 203 }  ,fill opacity=1 ] (297.36,109.76) .. controls (297.36,103.56) and (302.39,98.53) .. (308.59,98.53) -- (410.27,98.53) .. controls (416.48,98.53) and (421.5,103.56) .. (421.5,109.76) -- (421.5,143.45) .. controls (421.5,149.65) and (416.48,154.68) .. (410.27,154.68) -- (308.59,154.68) .. controls (302.39,154.68) and (297.36,149.65) .. (297.36,143.45) -- cycle ;

%Rounded Rect [id:dp6514528209760105] 
\draw  [color={rgb, 255:red, 255; green, 255; blue, 255 }  ,draw opacity=1 ][fill={rgb, 255:red, 255; green, 192; blue, 203 }  ,fill opacity=1 ] (428.62,11.8) .. controls (428.62,5.59) and (433.64,0.57) .. (439.85,0.57) -- (541.52,0.57) .. controls (547.73,0.57) and (552.75,5.59) .. (552.75,11.8) -- (552.75,45.49) .. controls (552.75,51.69) and (547.73,56.72) .. (541.52,56.72) -- (439.85,56.72) .. controls (433.64,56.72) and (428.62,51.69) .. (428.62,45.49) -- cycle ;
%Rounded Rect [id:dp5123602691575795] 
\draw  [color={rgb, 255:red, 255; green, 255; blue, 255 }  ,draw opacity=1 ][fill={rgb, 255:red, 255; green, 192; blue, 203 }  ,fill opacity=1 ] (572.5,110.52) .. controls (572.5,104.32) and (577.52,99.29) .. (583.73,99.29) -- (685.4,99.29) .. controls (691.61,99.29) and (696.63,104.32) .. (696.63,110.52) -- (696.63,144.21) .. controls (696.63,150.41) and (691.61,155.44) .. (685.4,155.44) -- (583.73,155.44) .. controls (577.52,155.44) and (572.5,150.41) .. (572.5,144.21) -- cycle ;

% Text Node
\draw (-0.77,114.61) node [anchor=north west][inner sep=0.75pt]    {$F_{1}^{1}$};
% Text Node
\draw (274.14,115.83) node [anchor=north west][inner sep=0.75pt]    {$F_{2}^{2}$};
% Text Node
\draw (125.06,21.1) node [anchor=north west][inner sep=0.75pt]    {$F_{1}^{2}$};
% Text Node
\draw (397.38,21.33) node [anchor=north west][inner sep=0.75pt]    {$F_{2}^{3}$};
% Text Node
\draw (547.63,116.12) node [anchor=north west][inner sep=0.75pt]    {$F_{3}^{3}$};
% Text Node
\draw (451.84,80.22) node [anchor=north west][inner sep=0.75pt]   [align=left] {$\displaystyle f_{2,3}^{2}$};
% Text Node
\draw (174.49,80.73) node [anchor=north west][inner sep=0.75pt]   [align=left] {$\displaystyle f_{1,2}^{1}$};
% Text Node
\draw (240.12,79.7) node [anchor=north west][inner sep=0.75pt]   [align=left] {$\displaystyle f_{1,2}^{2}$};
% Text Node
\draw (517.47,80.73) node [anchor=north west][inner sep=0.75pt]   [align=left] {$\displaystyle f_{2,3}^{3}$};
% Text Node
\draw (26.01,131.07) node [anchor=north west][inner sep=0.75pt]   [align=left] {$\displaystyle a_{1}$};
% Text Node
\draw (32.73,108.29) node [anchor=north west][inner sep=0.75pt]   [align=left] {$\displaystyle a_{2}$};
% Text Node
\draw (128.93,128.87) node [anchor=north west][inner sep=0.75pt]   [align=left] {$\displaystyle c$};
% Text Node
\draw (96.94,102.87) node [anchor=north west][inner sep=0.75pt]   [align=left] {$\displaystyle b$};
% Text Node
\draw (163.97,33.73) node [anchor=north west][inner sep=0.75pt]   [align=left] {$\displaystyle a_{1}$};
% Text Node
\draw (170.69,10.95) node [anchor=north west][inner sep=0.75pt]   [align=left] {$\displaystyle a_{2}$};
% Text Node
\draw (266.88,31.53) node [anchor=north west][inner sep=0.75pt]   [align=left] {$\displaystyle c$};
% Text Node
\draw (303.88,130.67) node [anchor=north west][inner sep=0.75pt]   [align=left] {$\displaystyle b'$};
% Text Node
\draw (310.56,107.9) node [anchor=north west][inner sep=0.75pt]   [align=left] {$\displaystyle a^{*}$};
% Text Node
\draw (406.76,128.47) node [anchor=north west][inner sep=0.75pt]   [align=left] {$\displaystyle c$};
% Text Node
\draw (435.13,32.71) node [anchor=north west][inner sep=0.75pt]   [align=left] {$\displaystyle b'$};
% Text Node
\draw (441.81,9.94) node [anchor=north west][inner sep=0.75pt]   [align=left] {$\displaystyle a^{*}$};
% Text Node
\draw (681.82,129.23) node [anchor=north west][inner sep=0.75pt]   [align=left] {$\displaystyle c'$};
% Text Node
\draw (585.69,108.66) node [anchor=north west][inner sep=0.75pt]   [align=left] {$\displaystyle a^{*}$};
% Text Node
\draw (579.01,131.43) node [anchor=north west][inner sep=0.75pt]   [align=left] {$\displaystyle b'$};

\end{tikzpicture}
}
\end{equation}

\end{note}

\subsection{Temporal Analogues of Static Properties}\label{sec:lifting-properties}

The theory of static data (be it graph theory, group theory, etc.) is far better understood than its temporal counterpart (temporal graphs, temporal groups, etc.). For this reason and since static properties are often easier to think of, it is natural to try to lift notions from the static setting to the temporal. 

This idea has been employed very often in temporal graph theory for instance with the notion of a \textit{temporal path}. In this section we will consider temporal paths and their definition in terms of graph narratives. This section is a case-study intended to motivate our more general approach in Section~\ref{sec:lifting-properties}.

\subsubsection{Temporal Paths} As we mentioned in Section~\ref{sec:desiderata}, one easy way of defining the notion of a temporal path in a temporal graph \(\Ga\) is to simply declare it to be a path in the underlying static graph of \(\Ga\). However, at first glance (and we will address this later on) this notion does not seem to be particularly `temporal' since it is forgetting entirely the various temporal relationships between edges and vertices. In contrast (using Kempe et. al.'s Definition~\ref{def:simple_definition_of_temporal_graph} of a temporal graph) temporal paths are usually defined as follows (we say that these notions are `\(\mathbf{(K3)}\)-temporal' to make it clear that they are defined in terms of Kempe, Kleinberg and Kumar's definition of a temporal graph). 

\begin{definition}[\(\mathbf{(K3)}\)-temporal paths and walks]\label{def:temporal-path}
Given vertices $x$ and $y$ in a temporal graph $(G, \tau)$, a \define{temporal $(x,y)$-walk} is a sequence $W = (e_1,t_1), \dots, (e_n,t_n)$ of edge-time pairs such that $e_1, \dots, e_n$ is a walk in $G$ starting at $x$ and ending at $y$ and such that \(e_i\) is active at time \(t_i\) and $t_1 \leq t_2 \leq \dots \leq t_n$. We say that a temporal $(x,y)$-walk is \define{closed} if $x = y$ and we say that it is \define{strict} if the times of the walk form a strictly increasing sequence.
\end{definition}

Using this definition, one also has the following natural decision problem on temporal graphs. 

\begin{framed}
    \noindent \(Temp_{K^3}Path_n\) \\
    \noindent \textbf{Input:} a \(\mathbf{(K3)}\)-temporal graph \(G := (V, (E_i)_{i\in\mathbb{N}})\) and an \(n \in \mathbb{N}\) \\
    \noindent \textbf{Task:} determine if there exists a \(\mathbf{(K3)}\)-temporal path of length at least \(n\) in \(G\).
\end{framed}

Notice that in static graph theory most computational problems can be cast as homomorphism problems in appropriate categories of graphs. For instance, the question of determining whether a fixed graph \(G\) admits a path of length at least \(n\) is equivalent to asking if there is at least one injective homomorphism \(P_n \hookrightarrow G\) from the \(n\)-path to \(G\). Similarly, if we wish to ask if \(G\) contains a clique on \(n\) vertices as a minor\footnote{Recall that a \define{contraction} of a graph \(G\) is a surjective graph homomorphism \(q \colon G \twoheadrightarrow G'\) such that every preimage of \(q\) is connected in \(G\) (equivalently \(G'\) is obtained from \(G\) by a sequence of edge contractions). A \define{minor} of a graph \(G\) is a subgraph \(H\) of a \define{contraction} \(G'\) of \(G\).}, then this is simply a homomorphism problem in the category \(\cat{Grph}_{\preceq}\) having graphs as objects and graph minors as morphisms: \(G\) contains \(K_n\) as a minor if and only if the hom-set \(\cat{Grph}_{\preceq}(K_n, G)\) is nonempty. 

Wishing to emulate this pattern from traditional graph theory, one immediately notices that, in order to define notions such as temporal paths, cliques and colorings (to name but a few), one first needs two things: 
\begin{enumerate}
    \item a notion of morphism of temporal graphs and 
    \item a way of lifting graph classes to classes of temporal graphs (for instance defining temporal path-graphs, temporal complete graphs, etc...). 
\end{enumerate}
Fortunately our narratives come equipped with a notion of morphism (these are simply natural transformations between the functors encoding the narratives). Thus, all that remains to be determined is how to convert classes of graphs into classes of temporal graphs. More generally we find ourselves interested in converting classes of objects of any category \(\cat{C}\) into classes of \(\cat{C}\)-narratives. We will address these questions in an even more general manner (Propositions~\ref{prop:change-of-base-covariant} and~\ref{prop:change-of-base-contravariant}) by developing a systematic way for converting \(\cat{C}\)-narratives into \(\cat{D}\)-narratives whenever we have certain kinds of data-conversion functors \(K \colon \cat{C} \to \cat{D}\). 

\begin{proposition}[Covariant Change of base]\label{prop:change-of-base-covariant}
Let \(\cat{C}\) and  \(\cat{D}\) be categories with limits (resp. colimits) and let \(\cat{T}\) be any time category. If \(K \colon \cat{C} \to \cat{D}\) is a continuous functor, then composition with \(K\) determines  a functor \((K \circ -)\) from persistent (resp. cumulative) \(\cat{C}\)-narratives to persistent (resp. cumulative) \(\cat{D}\)-narratives. Spelling this out explicitly for the case of persistent narratives, we have:
\begin{align*}
    (K \circ -) &\colon \cat{Pe}(\cat{T}, \cat{C}) \to \cat{Pe}(\cat{T}, \cat{D}) \\
    (K \circ -) &\colon (F \colon \cat{T}^{op} \to \cat{C}) \mapsto (K \circ F \colon \cat{T}^{op} \to \cat{D}).
\end{align*}
\end{proposition}
\begin{proof}
It is standard to show that \(K \circ F\) is a functor of presheaf categories, so all that remains is to show that it maps any \(\cat{C}\)-narrative \(F \colon \cat{T}^{op} \to \cat{C}\) to an appropriate sheaf. This follows immediately since \(K\) preserves limits: for any cover \(([a,p], [p, b])\) of any interval \([a,b]\) we have \((K \circ F)([a,b])) = K(F([a,p]) \times_{F([p,p])} F([p,b])) = (K \circ F)([a,p]) \times_{(K \circ F)([p,p])} (K \circ F)([p,b])).\) By duality the case of cumulative narratives follows.
\end{proof}

Notice that one also has change of base functors for any contravariant functor \(L \colon \cat{C}^{op} \to \cat{D}\) taking limits in \(\cat{C}\) to colimits in \(\cat{D}\). This yields the following result (which can be proven in the same way as Proposition~\ref{prop:change-of-base-covariant}).

\begin{proposition}[Contravariant Change of base]\label{prop:change-of-base-contravariant}
Let \(\cat{C}\) be a category with limits (resp. colimits) and \(\cat{D}\) be a category with colimits (resp. limits) and let \(\cat{T}\) be any time category. If \(K \colon \cat{C}^{op} \to \cat{D}\) is a functor taking limits to colimits (resp. colimits to limits), then the composition with \(K\) determines  a functor from persistent (resp. cumulative) \(\cat{C}\)-narratives to cumulative (resp. persistent) \(\cat{D}\)-narratives. 
\end{proposition}

To see how these change of base functors are relevant to lifting classes of objects in any category \(\cat{C}\) to corresponding classes of \(\cat{C}\)-narratives, observe that any such class \(\cat{P}\) of objects in \(\cat{C}\) can be identified with a subcategory \(P \colon \cat{P} \to \cat{C}\). One should think of this as a functor which picks out those objects of \(\cat{C}\) that satisfy a given property \(P\). Now, if this functor \(P\) is continuous, then we can apply Proposition~\ref{prop:change-of-base-covariant} to identify a class
\begin{equation}\label{eqn:change-of-base}
    (P \circ - ) \colon \cat{Pe}(\cat{T}, \cat{P}) \to \cat{Pe}(\cat{T}, \cat{C})
\end{equation} of \(\cat{C}\)-narratives which satisfy the property \(P\) at all times. Similar arguments let us determine how to specify temporal analogues of properties under the cumulative perspective. For example, consider the full subcategory \(\mathfrak{P} \colon \cat{Paths} \hookrightarrow \cat{Grph}\) which defines the category of all paths and the morphisms between them. As the following proposition shows, the functor \(\mathfrak{P}\) determines a subcategory \(\cat{Cu}(\cat{T}, \cat{Paths}) \hookrightarrow \cat{Cu}(\cat{T}, \cat{Grph})\) whose objects are temporal path-graphs. 

\begin{proposition}\label{prop:def:temporal-path}
The monic cosheaves in \(\cat{Cu}(\cat{T}, \cat{Paths})\) determine temporal graphs (in the sense of Definition~\ref{def:simple_definition_of_temporal_graph}) whose underlying static graph over any interval of time is a path. Furthermore, for any graph narrative \(\Ga \in \cat{Cu}(\cat{T}, \cat{Grph})\) all of the temporal paths in \(\Ga\) assemble into a poset \(\mathsf{Sub}_{(\mathfrak{P} \circ -)}(\Ga)\) defined as the subcategory of the subobject category \(\mathsf{Sub}(\Ga)\) whose objects are in the range of \((\mathfrak{P} \circ -)\). Finally, strict temporal paths in a graph narrative \(\Ga\) consists of all those monomorphism \(\mathfrak{P}(\Pa) \hookrightarrow \Ga\) where the path narrative \(\Pa\) in \(\mathsf{Sub}_{(\mathfrak{P} \circ -)}(\Ga)\) sends each instantaneous interval (i.e. one of the form \([t,t]\)) to a single-edge path.
\end{proposition}
\begin{proof}
Since categories of copresheaves are adhesive~\cite{LackAdhesive} (thus their pushouts preserve monomorphims), one can verify that, \textit{when they exists} (pushouts of paths need not be paths in general), pushouts in \(\cat{Paths}\) are given by computing pushouts in \(\cat{Grph}\). Thus a monic cosheaf \(\Pa\) in \(\cat{Cu}(\cat{T}, \cat{Paths})\) is necessarily determined by paths for each interval of time that combine (by pushout) into paths at longer intervals, as desired. Finally, by noticing that monomorphisms of (co)sheaves are simply natural transformations whose components are all monic, one can verify that any monomorphism from \(\mathfrak{P}(\Pa)\) to \(\Ga\) in the category of graph narratives determines a temporal path of \(\Ga\) and that this temporal path is strict if \(\Pa([t,t])\) is a path on at most one edge for all \(t \in \cat{T}\). Finally, as is standard in category theory~\cite{awodey2010category}, observe that one can collect all such monomorphisms (varying \(\Pa\) over all objects of \(\cat{Cu}(\cat{T}, \cat{Paths})\)) into a subposet of the subobject poset of \(\Ga\), which, by our preceding observation, determines all of the temporal paths in \(\Ga\).  
\end{proof}

\paragraph{Comparing the Cumulative to the Persistent.} Given Proposition~\ref{prop:def:temporal-path} one might wonder what a temporal path looks like under the \textit{persistent} perspective. By duality (and since pullbacks preserve monomorphisms and connected subgraphs of paths are paths) one can see that \textit{monic} persistent path narratives must consist of paths at each snapshot satisfying the property that over any interval the data persisting over that interval is itself a path. 

Since applying the functor \(\mathscr{P} \colon \cat{Cu}(\cat{T}, \cat{Paths}) \to \cat{Pe}(\cat{T}, \cat{Paths})\) of Theorem~\ref{thm:adjunction} turns any cumulative path narrative into a persistent one, it seem at first glance that there is not much distinction between persistent temporal paths and those defined cumulatively in Proposition~\ref{prop:def:temporal-path}. However, the distinction becomes apparent once one realises that in general we cannot simply turn a persistent path narrative into a cumulative one: in general arbitrary pushouts of paths need not be paths (they can give rise to trees). 

Realizing the distinctions between cumulative and persistent paths is a pedagogical example of a subtlety that our systematic approach to the study of temporal data can uncover but that would otherwise easily go unnoticed: in short, this amounts to the fact that studying the problem of the temporal tree (defined below) is equivalent to studying the persistent temporal path problem.

To make this idea precise, consider the adjunction 
% https://q.uiver.app/#q=WzAsMixbMCwwLCJcXG1hdGhzZntQZX0oXFxjYXR7VH0sIFxcbWF0aHNme0dycGhfe21vbm99fSkiXSxbMiwwLCJcXG1hdGhzZntDdX0oXFxjYXR7VH0sIFxcbWF0aHNme0dycGhfe21vbm99fSkiXSxbMCwxLCJ7XFxtYXRoc2Nye0t9fSIsMCx7ImN1cnZlIjotMn1dLFsxLDAsIntcXG1hdGhzY3J7UH19IiwwLHsiY3VydmUiOi0yfV0sWzIsMywiIiwyLHsibGV2ZWwiOjEsInN0eWxlIjp7Im5hbWUiOiJhZGp1bmN0aW9uIiwiYm9keSI6eyJuYW1lIjoibm9uZSJ9LCJoZWFkIjp7Im5hbWUiOiJub25lIn19fV1d
\[\begin{tikzcd}
	{\mathsf{Pe}(\cat{T}, \mathsf{Grph_{mono}})} && {\mathsf{Cu}(\cat{T}, \mathsf{Grph_{mono}})}
	\arrow[""{name=0, anchor=center, inner sep=0}, "{{\mathscr{K}}}", curve={height=-12pt}, from=1-1, to=1-3]
	\arrow[""{name=1, anchor=center, inner sep=0}, "{{\mathscr{P}}}", curve={height=-12pt}, from=1-3, to=1-1]
	\arrow["\dashv"{anchor=center, rotate=-90}, draw=none, from=0, to=1]
\end{tikzcd}\]

given to us by Theorem~\ref{thm:adjunction} (notice that the result applies since \(\cat{Grph}\) has all limits and colimits). This together with Proposition~\ref{prop:change-of-base-covariant} applied to the full subcategory \(\mathfrak{T} \colon \cat{Trees_{mono}} \to \cat{Grph_{mono}}\) yields the following diagram.  
% https://q.uiver.app/#q=WzAsNCxbMCwwLCJcXG1hdGhzZntDdX0oXFxtYXRoc2Z7VH0sIFxcbWF0aHNme1RyZWVzfV97bW9ub30pIl0sWzIsMiwiXFxtYXRoc2Z7UGV9KFxcbWF0aHNme1R9LCBcXG1hdGhzZntHcnBofV97bW9ub30pIl0sWzAsMiwiXFxtYXRoc2Z7UGV9KFxcbWF0aHNme1R9LCBcXG1hdGhzZntQYXRoc31fe21vbm99KSJdLFsyLDAsIlxcbWF0aHNme0N1fShcXG1hdGhzZntUfSwgXFxtYXRoc2Z7R3JwaH1fe21vbm99KSJdLFsyLDEsIihcXG1hdGhmcmFre1B9IFxcY2lyYyAtKSJdLFsxLDMsIlxcbWF0aHNjcntLfSJdLFswLDMsIihcXG1hdGhmcmFre1R9IFxcY2lyYyAtKSJdXQ==
\[\begin{tikzcd}
	{\mathsf{Cu}(\mathsf{T}, \mathsf{Trees_{mono}})} && {\mathsf{Cu}(\mathsf{T}, \cat{Grph_{mono}})} \\
	\\
	{\mathsf{Pe}(\mathsf{T}, \cat{Paths_{mono}})} && {\mathsf{Pe}(\mathsf{T}, \cat{Grph_{mono}})}
	\arrow["{(\mathfrak{P} \circ -)}", from=3-1, to=3-3]
	\arrow["{\mathscr{K}}", from=3-3, to=1-3]
	\arrow["{(\mathfrak{T} \circ -)}", from=1-1, to=1-3]
\end{tikzcd}\]
The pullback (in \(\cat{Cat}\)) of this diagram (i.e. of $(\mathfrak{T} \circ -)$ and $ \mathscr{K} \circ (\mathfrak{P} \circ -)$) yields a category having as objects pairs \((\Ta, \Pa)\) consisting of a cumulative tree narrative \(\Ta\) and a persistent path narrative \(\Pa\) such that, when both are viewed as cumulative \(\cat{Grph_{mono}}\)-narratives, they give rise to the \textit{same} narrative. %Since the adjunction of Theorem~\ref{thm:adjunction} restricts to an \textit{equivalence} of categories, we have the 
Thus we see that the question of determining whether a cumulative graph narrative \(\Ga\) contains \(\mathfrak{T}(\Ta)\) as a sub-narrative can be reduced to the question of determining whether \(\Pa\) is a persistent path sub-narrative of \(\mathscr{P}(\Ga)\). 

\begin{aside}
Although it is far beyond the scope of this paper, we believe that there is a wealth of understanding of temporal data (and in particular temporal graphs) to be gained from the interplay of lifting graph properties and the persistent-cumulative adjunction of Theorem~\ref{thm:adjunction}. For instance the preceding discussion shows that one can equivalently study persistent paths instead of thinking about cumulative temporal trees. Since persistent paths are arguably easier to think about (because paths are fundamentally simpler objects than trees) it would stand to reason that this hidden connection between these classes of narratives could aid in making new observations that have so far been missed.
\end{aside}
 
\subsubsection{Changing the Resolution of Temporal Analogues.}

As we have done so far, imagine collecting data over time from some hidden dynamical system and suppose, after some exploratory analysis of our data, that we notice the emergence of some properties in our data that are only visible at a certain temporal resolution. For example it might be that some property of interest is only visible if we accumulate all of the data we collected over time intervals whose duration is at least ten seconds. 

In contrast notice that the temporal notions obtained solely by `change of base' (i.e. via functors such as \eqref{eqn:change-of-base}) are very strict: not only do they require each instantaneous snapshot to satisfy the given property \(P\), they also require the property to be satisfied by any data that persists (or, depending on the perspective, accumulates) over time. For instance the category of temporal paths of Proposition~\ref{prop:def:temporal-path} consists of graph narratives that are paths at \textit{all} intervals. In this section we will instead give a general, more permissive definition of temporal analogues or static notions. This definition will account for the fact that one is often only interested in properties that emerge at certain temporal resolutions, but not necessarily others. 

To achieve this, we will briefly explain how to functorially change the temporal resolution of our narratives (Proposition~\ref{prop:change-temp-resolution}). Then, combining this with our change of base functors (Propositions~\ref{prop:change-of-base-covariant} and~\ref{prop:change-of-base-contravariant}) we will give an extremely general definition of a temporal analogue of a static property. The fact that this definition is parametric in the temporal resolution combined with the adjunction that relates cumulative and persistent narratives  (Theorem~\ref{thm:adjunction}) leads to a luscious landscape of temporal notions whose richness can be systematically studied via our category-theoretic perspective.

\begin{proposition}[Change of Temporal Resolution]\label{prop:change-temp-resolution}
Let \(\cat{T}\) be a time category and \(\cat{S} \xhookrightarrow{\tau} \cat{T}\) be a sub-join-semilattice thereof. Then, for any category \(\cat{C}\) with (co)limits, there is a functor \((- \circ \tau)\) taking persistent (resp. cumulative) \(\cat{C}\) narratives with time \(\cat{T}\) to narratives of the same kind with time \(S\).
\end{proposition}
\begin{proof}
By standard arguments the functor is defined by post composition as \[
    (- \circ \tau) \colon \cat{Cu}(\cat{T}, \cat{C}) \to \cat{Cu}(\cat{S}, \cat{C}) 
    \quad \text{ where } \quad  
    (- \circ \tau) \colon \bigl( \Fa \colon \cat{T} \to \cat{C}\bigr) \mapsto \bigl( \Fa \circ \tau \colon \cat{S} \to \cat{C} \bigr).
\]
The persistent case is defined in the same way. 
\end{proof}

Thus, given a sub-join-semilattice \(\tau\colon S\hookrightarrow \cat{T}\) of some time-category \(\cat{T}\), we would like to specify the collection of objects of a category of narratives that satisfy some given property \(P\) only over the intervals in \(S\). A slick way of defining this is via a pullback of functors as in the following definition. 

\begin{definition}\label{def:lifting-properties}
Let \(\tau\colon \cat{S} \hookrightarrow \cat{T}\) be a sub-join-semilattice of a time category \(\cat{T}\) let \(\cat{C}\) be a category with limits and let \(P \colon \cat{P} \hookrightarrow \cat{C}\) be a continuous functor. Then we say that a persistent \(\cat{C}\)-narrative with time \(\cat{T}\) \define{\(\tau\)-satisfies the property \(P\)} if it is in the image of the pullback (i.e. the red, dashed functor in the following diagram) of \((-\circ \tau)\) along \((P \circ - \circ \tau)\). An analogous definition also holds for cumulative narratives when \(\cat{C}\) has colimits and \(\cat{P}\) is continuous.
% https://q.uiver.app/#q=WzAsNSxbMCwwLCJcXG1hdGhzZntQZX0oXFxtYXRoc2Z7VH0sIFxcbWF0aHNme1B9KSJdLFsyLDAsIlxcbWF0aHNme1BlfShcXG1hdGhzZntTfSwgXFxtYXRoc2Z7UH0pIl0sWzQsMCwiXFxtYXRoc2Z7UGV9KFxcbWF0aHNme1N9LCBcXG1hdGhzZntDfSkiXSxbNCwxLCJcXG1hdGhzZntQZX0oXFxtYXRoc2Z7VH0sIFxcbWF0aHNme0N9KSJdLFswLDEsIlxcbWF0aHNme1BlfShcXG1hdGhzZntUfSwgXFxtYXRoc2Z7Q30pIFxcdGltZXNfe1xcbWF0aHNme1BlfShcXG1hdGhzZntTfSwgXFxtYXRoc2Z7UH0pfSBcXG1hdGhzZntQZX0oXFxtYXRoc2Z7VH0sIFxcbWF0aHNme1B9KSJdLFsxLDIsIihQXFxjaXJjIC0pIl0sWzAsMSwiKC1cXGNpcmNcXHRhdSkiXSxbMywyLCIoLVxcY2lyY1xcdGF1KSIsMl0sWzQsMywiIiwwLHsiY29sb3VyIjpbMCw2MCw2MF0sInN0eWxlIjp7ImJvZHkiOnsibmFtZSI6ImRhc2hlZCJ9fX1dLFs0LDBdLFs0LDYsIiIsMix7ImxldmVsIjoxLCJzdHlsZSI6eyJuYW1lIjoiY29ybmVyIn19XV0=
\[\begin{tikzcd}[row sep=small]
	{\mathsf{Pe}(\mathsf{T}, \mathsf{P})} && {\mathsf{Pe}(\mathsf{S}, \mathsf{P})} && {\mathsf{Pe}(\mathsf{S}, \mathsf{C})} \\
	{\mathsf{Pe}(\mathsf{T}, \mathsf{C}) \times_{\mathsf{Pe}(\mathsf{S}, \mathsf{P})} \mathsf{Pe}(\mathsf{T}, \mathsf{P})} &&&& {\mathsf{Pe}(\mathsf{T}, \mathsf{C})}
	\arrow["{(P\circ -)}", from=1-3, to=1-5]
	\arrow[""{name=0, anchor=center, inner sep=0}, "{(-\circ\tau)}", from=1-1, to=1-3]
	\arrow["{(-\circ\tau)}"', from=2-5, to=1-5]
	\arrow[draw={rgb,255:red,214;green,92;blue,92}, dashed, from=2-1, to=2-5]
	\arrow[from=2-1, to=1-1]
	\arrow["\lrcorner"{anchor=center, pos=0.125, rotate=90}, draw=none, from=2-1, to=0]
\end{tikzcd}\]
\end{definition}

As a proof of concept, we shall see how Definition~\ref{def:lifting-properties} can be used to recover notions of temporal cliques as introduced by Viard, Latapy and Magnien~\cite{viard2016computing}. 

Temporal cliques were thought of as models of groups of people that commonly interact with each other within temporal contact networks. Given the apparent usefulness of this notion in epidemiological modeling and since the task of finding temporal cliques is algorithmically challenging, this notion has received considerable attention recently~\cite{hermelin2023temporal, banerjee2019enumeration, bentert2019listing, himmel2017adapting, molter2021isolation, viard2014identifying}. They are typically defined in terms of Kempe, Kleinberg and Kumar's definition of a temporal graph (Definition~\ref{def:simple_definition_of_temporal_graph}) (or equivalently in terms of \textit{link streams}) where one declares a temporal clique to be a vertex subset \(S\) of the time-invariant vertex set such that, cumulatively, over any interval of length at least some given \(k\), \(S\) induces a clique. The formal definition follows. 
\begin{definition}[\cite{viard2016computing}]\label{def:temporal-clique}
Given a \(\mathbf{(K3)}\)-temporal graph \(G := (V, (E_i)_{i\in\mathbb{N}})\) and an \(n \in \mathbb{N}\), a subset \(S\) of \(V\) is said to be a \define{temporal \(k\) clique} if \(|S|\geq k\) and if for all intervals \([a, b]\) of length \(n\) in \(\mathbb{N}\) (i.e. \(b = a + n -1\)) one has that: for all \(x,y\in S\) there is an edge incident with both \(x\) and \(y\) in \(\bigcup_{t \in [a, b]} E_t\).
\end{definition}

Now we will see how we can obtain the above definition as an instance of our general construction of Definition~\ref{def:lifting-properties}. We should note that the following proposition is far more than simply recasting a known definition into more general language. Rather, it is about simultaneously achieving two goals at once. 
\begin{enumerate}
    \item It is showing that the instantiation of our general machinery (Definition~\ref{def:lifting-properties}) recovers the specialized definition (Definition~\ref{def:temporal-clique}).
    \item It provides an alternative characterization of temporal cliques in terms of \textit{morphisms} of temporal graphs. This generalizes the traditional definitions of cliques in static graph theory as injective homomorphisms into a graph from a complete graph. 
\end{enumerate}

\begin{proposition}\label{prop:temporal-cliques}
Let \(\kappa_{\geq k} \colon \cat{Complete}^{\geq k} \hookrightarrow \cat{Grph}\) be the subcategory of \(\cat{Grph}\) whose objects are complete graphs on at least \({k}\) vertices and let \(\tau_{\geq n} \colon \cat{T} \to \cat{I}_\mathbb{N}\) be the sub-join-semilattice of \(\cat{I}_\mathbb{N}\) whose objects are intervals of \(\cat{T}_\mathbb{N}\) length at least \(n\). Consider any graph narrative \(\Ka\) which \(\tau_n\)-satisfies \(\kappa_{\geq k}\) then all of its instantaneous snapshots \(\Ka([t,t])\) have at least \(k\) vertices. Furthermore consider any monomorphism \(f \colon \Ka \hookrightarrow \Ga\) from such a \(\Ka\) to any given cumulative graph narrative \(\Ga\). If \(\Ka\) preserves monomorphisms, then we have that: every such morphism of narratives \(f\) determines a temporal clique in \(\Ga\) (in the sense of Definition~\ref{def:temporal-clique}) and moreover all temporal cliques in \(\Ga\) are determined by morphisms of this kind. 
\end{proposition}
\begin{proof}
First of all observe that if a pushout \(L +_M R\) of a span graphs \(L \xleftarrow{\ell} M \xrightarrow{r} R\) is a complete graph, then we must have that at least one of the graph homomorphisms \(\ell\) and \(r\) must be surjective on the vertex set (if not then there would be some vertex of \(L\) not adjacent to some vertex of \(R\) in the pushout). With this in mind now consider any cumulative graph narrative \(\Ka\) which \(\tau_{\geq n}\)-satisfies \(\kappa_{\geq k}\). By Definition~\ref{def:lifting-properties} this means that for all intervals \([a,b]\) of length at least \(n\) the graph \(\Ka([a,b])\) is in the range of \(\kappa_{\geq k}\): i.e. it is a complete graph on at least \(k\) vertices. This combined with the fact that \(\Ka\) is a cumulative narrative implies that every pushout of the form \[\Ka([a,p]) +_{\Ka([p,p])} \Ka([p,b])\] yields a complete graph and hence every pair of arrows \[\Ka([a,p]) \xleftarrow{\ell} \Ka([p,p]) \xrightarrow{r} \Ka([p,b])\] must have at least one of \(\ell\) or \(r\) surjective. From this one deduces that for all times \(t \geq n\) every instantaneous graph \(\Ka([t,t])\)  must have at least \(k\) vertices: since \(\Ka\) \(\tau_{\geq n}\)-satisfies \(\kappa_{\geq k}\), the pushout of the span \[\Ka([t - n + 1, t]) +_{\Ka([t,t])} \Ka([t, t + n - 1])\] must be a complete graph on at least \(k\) vertices and this is also true of both feet of this span, thus we are done by applying the previous observation. 

Observe that, if \(S\) is a vertex set in \(\Ga\) which determines a temporal clique in the sense of Definition~\ref{def:temporal-clique}, then this immediately determines a cumulative graph narrative \(\Ka\) which \(\tau_{\geq n}\)-satisfies \(\kappa_{\geq k}\) and that has a monomorphism into \(\Ga\): for any interval \([a,b]\), \(\Ka([a,b])\) is defined as the restriction (i.e. induced subgraph) of \(\Ga([a,b])\) to the vertices in \(S\). The fact that \(\Ka\) preserves monomorphisms follows since \(\Ga\) does. 

For the converse direction, notice that, if \(\Ka\) preserves monomorphisms (i.e. the projection maps of its cosheaf structure are monomorphisms), then, by what we just argued, for any interval \([a,b]\) we have \(|\Ka([a,b])| \geq |\Ka([a,a])| \geq k\). Thus, since all of the graphs of sections have a lower bound on their size, we have that there must exist some time \(t\) such that \(\Ka([t, t+n -1])\) has minimum number of vertices. We claim that the vertex-set of \(\Ka([t, t+n-1])\) defines a temporal clique in \(\Ga\) (in the sense of Definition~\ref{def:temporal-clique}). To that end, all that we need to show is that the entire vertex set of \(\Ka([t, t+n-1])\) is active in every interval of length exactly \(n\). To see why, note that, since all of the projection maps in the cosheaf \(\Ka\) are monic, every interval of length at least \(n\) will contain all of the vertex set of \(\Ka([t, t+n-1])\); furthermore each pair of vertices will be connected by at least one edge in the graphs associated to such intervals since \(\Ka\) \(\tau_{\geq n}\)-satisfies \(\kappa_{\geq k}\). 

Thus, to conclude the proof, it suffices to show that for all times \(s \geq n - 1\) we have that every vertex of \(\Ka([t, t+n-1])\) is contained in \(\Ka([s,s])\) (notice that for smaller \(s\) there is nothing to show since there is no interval \([s',s]\) of length at least \(n\) which needs to witness a clique on the vertex set of \(\Ka([t, t+n-1])\)). To that end we distinguish three cases.
\begin{enumerate}
    \item Suppose \(s \not \in [t, t + n - 1]\), then, if \(s > t + n -1\), consider the diagram of monomorphisms
    % https://q.uiver.app/#q=WzAsNCxbMiwyLCJcXEthKFtzLHNdKSJdLFswLDEsIlxcS2EoW3QsdCArIG4gLSAxXSkiXSxbMywxLCJcXEthKFtzLHMgKyBuIC0gMV0pIl0sWzAsMCwiXFxLYShbdCxzXSkiXSxbMCwyLCJyIiwwLHsic3R5bGUiOnsidGFpbCI6eyJuYW1lIjoiaG9vayIsInNpZGUiOiJ0b3AifX19XSxbMCwzLCJcXGVsbCIsMix7InN0eWxlIjp7InRhaWwiOnsibmFtZSI6Imhvb2siLCJzaWRlIjoidG9wIn19fV0sWzEsMywiIiwwLHsic3R5bGUiOnsidGFpbCI6eyJuYW1lIjoiaG9vayIsInNpZGUiOiJ0b3AifX19XV0=
\[\begin{tikzcd}
	{\Ka([t,s])} \\
	{\Ka([t,t + n - 1])} &&& {\Ka([s,s + n - 1])} \\
	&& {\Ka([s,s])}
	\arrow["r", hook, from=3-3, to=2-4]
	\arrow["\ell"', hook, from=3-3, to=1-1]
	\arrow[hook, from=2-1, to=1-1]
\end{tikzcd}\]
and observe by our previous arguments that \(\ell\) or \(r\) must be surjective on vertices. We claim that \(\ell\) is always a vertex-surjection: if \(r\) is surjective on vertices, then, by  the minimality of the number of vertices of \(\Ka([t,t + n - 1])\) and the fact that the diagram is monic, we must have that \(\ell\) is surjective on vertices. But then this yields the desired result since we have a diagram of monomorphisms. Otherwise, if \(s < t\) either \(s < n - 1\) (in which case there is nothing to show), or a specular argument to the one we just presented for case of \(s > t + n - 1\) suffices. 
\item If \(s \in [t, t+n-1]\), then consider the following diagram
% https://q.uiver.app/#q=WzAsOCxbMiwxLCJcXEthKFtzLHNdKSJdLFsyLDAsIlxcS2EoW3QsdCArIG4gLSAxXSkiXSxbMywwLCJcXEthKFtzLHMgKyBuIC0gMV0pIl0sWzEsMCwiXFxLYShbcyAtIG4gKzEsc10pIl0sWzEsMSwiXFxLYShbdCx0XSkiXSxbMCwwLCJcXEthKFt0IC0gbiArMSx0XSkiXSxbMywxLCJcXEthKFt0ICsgbiAtIDEsIHQgKyBuIC0gMV0pIl0sWzQsMCwiXFxLYShbdCArIG4gLSAxLCB0ICsgMihuIC0gMSldKSJdLFswLDIsIlxcYmV0YSIsMSx7ImxhYmVsX3Bvc2l0aW9uIjoyMCwic3R5bGUiOnsidGFpbCI6eyJuYW1lIjoiaG9vayIsInNpZGUiOiJ0b3AifX19XSxbMCwxLCIiLDIseyJzdHlsZSI6eyJ0YWlsIjp7Im5hbWUiOiJob29rIiwic2lkZSI6InRvcCJ9fX1dLFswLDMsIlxcYWxwaGEiLDEseyJsYWJlbF9wb3NpdGlvbiI6MjAsInN0eWxlIjp7InRhaWwiOnsibmFtZSI6Imhvb2siLCJzaWRlIjoidG9wIn19fV0sWzQsNSwiIiwwLHsic3R5bGUiOnsidGFpbCI6eyJuYW1lIjoiaG9vayIsInNpZGUiOiJ0b3AifX19XSxbNCwxLCJmIiwxLHsibGFiZWxfcG9zaXRpb24iOjIwLCJzdHlsZSI6eyJ0YWlsIjp7Im5hbWUiOiJob29rIiwic2lkZSI6InRvcCJ9fX1dLFs2LDcsIiIsMSx7InN0eWxlIjp7InRhaWwiOnsibmFtZSI6Imhvb2siLCJzaWRlIjoidG9wIn19fV0sWzYsMSwiZyIsMSx7ImxhYmVsX3Bvc2l0aW9uIjoyMCwic3R5bGUiOnsidGFpbCI6eyJuYW1lIjoiaG9vayIsInNpZGUiOiJ0b3AifX19XSxbNCwzLCIiLDEseyJzdHlsZSI6eyJ0YWlsIjp7Im5hbWUiOiJob29rIiwic2lkZSI6InRvcCJ9fX1dLFs2LDIsIiIsMSx7InN0eWxlIjp7InRhaWwiOnsibmFtZSI6Imhvb2siLCJzaWRlIjoidG9wIn19fV1d
\[
\adjustbox{scale=1.5, max width=0.925\textwidth}{%,center}{
\begin{tikzcd}
	{\Ka([t - n +1,t])} & {\Ka([s - n +1,s])} & {\Ka([t,t + n - 1])} & {\Ka([s,s + n - 1])} & {\Ka([t + n - 1, t + 2(n - 1)])} \\
	& {\Ka([t,t])} & {\Ka([s,s])} & {\Ka([t + n - 1, t + n - 1])}
	\arrow["\beta"{description, pos=0.2}, hook, from=2-3, to=1-4]
	\arrow[hook, from=2-3, to=1-3]
	\arrow["\alpha"{description, pos=0.2}, hook, from=2-3, to=1-2]
	\arrow[hook, from=2-2, to=1-1]
	\arrow["f"{description, pos=0.2}, hook, from=2-2, to=1-3]
	\arrow[hook, from=2-4, to=1-5]
	\arrow["g"{description, pos=0.2}, hook, from=2-4, to=1-3]
	\arrow[hook, from=2-2, to=1-2]
	\arrow[hook, from=2-4, to=1-4]
\end{tikzcd}
}
\]
and observe that, by the same minimality arguments as in the previous point, we have that \(f\) and \(g\) must be surjective on vertices. By what we argued earlier, one of \(\alpha\) and \(\beta\) must be surjective on vertices; this combined with the fact that there are monomorphisms \[\Ka([t,t]) \hookrightarrow \Ka([s - n +1, s]) \text{ and } \Ka([t + n -1,t + n -1]) \hookrightarrow [s, s + n -1]\] (since \(t \in [s - n +1, s]\) and \(t + n -1 \in [s, s + n -1]\)) implies that every vertex of \(\Ka([t,t + n -1])\) is contained in \(\Ka([s,s])\) as desired. 
\end{enumerate}
\end{proof}

In the world of static graphs, it is well known that dual to the notion of a clique in a graph is that of a \textit{proper coloring}. This duality we refer to is not merely aesthetics, it is formal: if a clique in a graph \(G\) is a monomorphism from a complete graph \(K_n\) into \(G\), then a \textit{coloring} of \(G\) is a monomorphism \(K_n \hookrightarrow G\) in the \textit{opposite category}. Note that this highlights the fact that different categories of graphs give rise to different notions of coloring via this definition (for instance note that, although the typical notion of a graph coloring is defined in terms of \textit{irreflexive} graphs, the definition given above can be stated in any category of graphs). 

In any mature theory of temporal data and at the very least any theory of temporal graphs, one would expect there to be similar categorical dualities at play. And indeed there are: by dualizing Proposition~\ref{prop:temporal-cliques}, one can recover different notions of temporal coloring depending on whether one studies the cumulative or persistent perspectives. This is an illustration of a much deeper phenomenon whereby stating properties of graphs in a categorical way allows us to both lift them to corresponding temporal analogues while also retaining the ability to explore how they behave by categorical duality. 

\section{Discussion: Towards a General Theory of Temporal Data}\label{sec:towards-a-general-theory}
Here we tackled the problem of building a robust and general theory of temporal data. First we distilled a list of five desiderata (see \ref{desideratum-1}, \ref{desideratum-2}, \ref{desideratum-3}, \ref{desideratum-4}, \ref{desideratum-5} in Section~\ref{sec:intro}) for any such theory by drawing inspiration from the study of temporal graphs, a relatively well-developed branch of the mathematics of time-varying data. 

Given this list of desiderata, we introduced the notion of a \textit{narrative}. This is a kind of sheaf on a poset of intervals (a join-semilattice thereof, to be precise) which assigns to each interval of time an object of a given category and which relates the objects assigned to different intervals via appropriate restriction maps. The structure of a sheaf arises immediately from considerations on how to encode the time-varying nature of data, which is not specific to the kinds of mathematical object one chooses to study (Desideratum~\ref{desideratum-4}). This object-agnosticism allows us to use of a single set of definitions to think of time varying graphs or simplicial complexes or metric spaces or topological spaces or groups or beyond. We expect the systematic study of different application areas within this formalism to be a very fruitful line of future work. Examples abound, but, in favor of concreteness, we shall briefly mention two such ideas: 
\begin{itemize}
    \item The shortest paths problem can be categorified in terms of the free category functor~\cite{algebraic-path-problem}. Since this is an adjoint, it satisfies the continuity requirements to be a change of base functor (Proposition~\ref{prop:change-of-base-covariant}) and thus one could define and study temporal versions of the algebraic path problem (a vast generalization of shortest paths) by relating narratives of graphs to narratives of categories. 
    \item Metabolic networks are cumulative representations of the processes that determine the physiological and biochemical properties of a cell. These are naturally temporal objects since different reactions may occur at different times. Since reaction networks, one of the most natural data structures to represent chemical reactions, can be encoded as copresheaves~\cite{aduddell2023compositional}, one can study time varying reaction networks via appropriate narratives valued in these categories. 
\end{itemize}

Encoding temporal data via narratives equips us with a natural choice of morphism of temporal data, namely: morphism of sheaves. Thus we find that narratives assemble into \textit{categories} (Desideratum~\ref{desideratum-1}), a fact that allows us to leverage categorical duality to find that narratives come in two flavours (cumulative and persistent, Desideratum~\ref{desideratum-2} depending on how information is being tracked over time. In sufficiently nice categories, persistent and cumulative narratives are furthermore connected via an adjunction (Theorem~\ref{thm:adjunction}) which allows one to convert one description into the other. As is often the case in mathematics, we expect this adjunction to play an important role for many categories of narratives.

To be able to lift notions from static settings to temporal ones, we find that it suffices to first determine canonical ways to change the temporal resolution of narratives or to change the underlying categories in which they are valued. Both of these tasks can be achieved functorially (Propositions~\ref{prop:change-of-base-covariant} and~\ref{prop:change-of-base-contravariant} and Proposition~\ref{prop:change-temp-resolution}) and, embracing minimalism, one finds that they are all that is needed to develop a framework for the systematic lifting of static properties to their temporal counterparts~\ref{desideratum-3}. 

Finally, addressing Desideratum~\ref{desideratum-4}, we showed how to obtain change of base functors (Propositions~\ref{prop:change-of-base-covariant} and~\ref{prop:change-of-base-contravariant}) which allows for the conversion of narratives valued in one category to another. In the interest of a self-contained presentation, we focused on only one application of these functors; namely that of building a general machinery (Definition~\ref{def:lifting-properties}) capable of lifting the definition of a property from any category to suitable narratives valued in it. However, the change of base functors have more far reaching applications than this and should instead be thought of as tools for systematically relating different kinds of narratives arising from the same dynamical system. This line of enquiry deserves its own individual treatment and we believe it to be a fascinating new direction for future work. 

In so far as the connection between data and dynamical systems is concerned (Desideratum~\ref{desideratum-5}), our contribution here is to place both the theory of dynamical systems and the theory of temporal data on the same mathematical and linguistic footing. This relies on the fact that Schultz, Spivak and Vasilakopoulou's \textit{interval sheaves}~\cite{schultz2017temporal} provide an approach to dynamical systems which is very closely related (both linguistically and mathematically) to our notion of narratives: both are defined in terms of sheaves on categories of intervals. We anticipate that exploring this newfound mathematical proximity between the way one represents temporal data and the axiomatic approach for the theory of dynamical systems will be a very fruitful line of further research in the years to come.

\enlargethispage{20pt}

%\ack{We would like to thank Justin Curry for helpful discussions and for observing connections of our work to Topological Data Analysis. We also thank Kevin Carlson and an anonymous reviewer for their helpful suggestions and corrections. }

%%%%%%%%%% Insert bibliography here %%%%%%%%%%%%%%

\vskip2pc

\section*{Acknowledgment}
We would like to thank Justin Curry for helpful discussions and for observing connections of our work to topological data analysis. We also thank Kevin Carlson and anonymous reviewers for their helpful suggestions and corrections.

%%%%%%%%%% Insert bibliography here %%%%%%%%%%%%%%

\vskip2pc

\bibliography{biblio}
\bibliographystyle{acm}

\end{document}